\documentclass[11pt]{article}

\usepackage{subfigure}
\usepackage{color}
\usepackage[cmex10]{amsmath}
\usepackage{mathrsfs,amssymb,amsmath}
\usepackage{cases}
\usepackage{graphicx}
\usepackage{caption}

\usepackage{epstopdf}
\usepackage{hyperref}

\usepackage{float}
\usepackage{placeins}

\def\disp{\displaystyle}

\newcommand{\R}{{\mathbb R}}

\def\dref#1{(\ref{#1})}
\def\crr{\cr\noalign{\vskip-0.0mm}}

\usepackage{algorithm} 
\usepackage{algorithmic} 

\newtheorem{corollary}{Corollary}[section]
\newtheorem{remark}{Remark}[section]
\newtheorem{lemma}{Lemma}[section]

\newtheorem{theorem}{Theorem}[section]
\newtheorem{definition}{Definition}[section]

\newenvironment{proof}{{\bf {Proof.}}}{\hfill $\square$}
\allowdisplaybreaks[4]

\numberwithin{equation}{section}

\def\dref#1{(\ref{#1})}

\def\pt{\partial}

\def\ra{\rightarrow}

\def\s{\subseteq}

\def\e{\varepsilon}

\def\ol{\overline}

\def\bf{\textbf}
\def\pt{\partial}

\def\om{\omega}
\def\Om{\Omega}
\def\la{\lambda}
\def\al{\alpha}
\def\be{\beta}
\def\de{\delta}
\def\ga{\gamma}
\def\De{\Delta}

\def\si{\sigma}

\def\ts{\times}

\def\iy{\infty}

\def\f{\frac}

\def\se{\setminus}

\def\df{\mathrm d}

\def\wh{\widehat}

\def\esssup{\operatorname*{ess\ \! sup}}
\def\essinf{\operatorname*{ess\ \! inf}}

\def\mcD{\mathcal{D}}

	\DeclareMathOperator{\Div}{div}

	\DeclareMathOperator{\dist}{dist}

	\newcommand{\N}{\mathbb N}

\begin{document}
	

\title{{\bf   {\bf  Null Controllability for a Multi-Dimensional Degenerate
Parabolic Equation with Degenerated Interior Point}}\footnote{\small This work was carried out with the support of the
National Natural Science Foundation of China under grant  nos. 12131008 and U23B2033, and
National Key R\&D Program of China under grant no. 2024YFA1013101.}}

\author{ Dong-Hui Yang$^{a}$,  Bao-Zhu Guo$^{b}$\footnote{\small
The corresponding author. Email: bzguo@iss.ac.cn}, Jie Zhong$^{c}$
\\
$^a${\it School of Mathematics and Statistics, Central South University}\\
			{\it Changsha 410075, P.R.China}\\
$^b${\it Academy of Mathematics and Systems Science, Academia Sinica, Beijing 100190, China}\\
$^c${\it   Department of Mathematics}\\
{\it California State University Los Angeles, Los Angeles, 90032, USA}
}
\date{}

	\maketitle{}
	\thispagestyle{empty}
	\thispagestyle{empty}
	
	
\begin{abstract}

  In this study, we study the null controllability of a multi-dimensional degenerate parabolic equation characterized by a degenerate interior point. The control domain, which is an arbitrary inner region, does  not encompass the degenerate point. To tackle this problem, we adopt a new  approximation methodology. Specifically, we approximate the degenerate partial differential equations (PDEs) with a series of uniformly elliptic PDEs, notwithstanding their limited regularity. We then derive the Carleman estimate for these approximate uniformly parabolic equations and establish the observability inequality, which ultimately paves the way for demonstrating the null controllability of the system.

\vspace{0.3cm}

	
\vspace{0.3cm}
		

	\end{abstract}
	
	\section{Introduction}

 In this paper, we  consider null controllability of the following degenerated  equation:
\begin{equation}\label{12.14.1}
\begin{cases}
\partial_t\varphi-\Div(|x|^\alpha \nabla \varphi)=\chi_\omega f, &\text{in }Q, \\
\varphi=0, &\text{on }\partial Q, \\
\varphi(0)=\varphi_0, &\text{in }\Omega,
\end{cases}
\end{equation}
where $\alpha\in (0,2)$ is a given constant, $\Omega\subset\mathbb{R}^N$ is a bounded domain containing $0$ with $\partial\Omega\in C^3$, $\omega\subset\Omega$ is an open subset such that $\dist(0,\omega)>0$, $Q=\Omega\times (0,T)$ with $T>0$ being a constant, $\varphi_0\in L^2(\Omega)$ is the initial data, and $f\in L^2(Q)$ is the control.
Let $x_0\in\omega$ and $R\in \left(0,\min\{1, \sqrt{\frac{T}{2}}\}\right)$ be such that $\widehat{\omega}=B_{2R}(x_0)\Subset \omega$, that is,
the closure of $\widehat{\omega}$ is compact and contained in $\omega$, and $B_{8R}\subset \Omega$.
Denote
\begin{equation}\label{gbz1}
w=|x|^\alpha.
\end{equation}

	 The controllability of partial differential equations (PDEs) has been extensively studied in the literature. For uniformly elliptic or parabolic PDEs, readers are referred to \cite{FG,Fursikov,Lions,Rousseau,Zuazua}. Regarding degenerate parabolic (or hyperbolic) equations, numerous works have focused on the one-dimensional case, such as \cite{Alabau,Buffe,Cannarsa,Gueye}, while several studies have addressed higher-dimensional scenarios, including \cite{Araruna,Cannarsa1,Guo,Wu}. Notably, \cite{Araruna,Cannarsa1} specifically investigated the two-dimensional case. The works \cite{Wu, Yang1} serve as precursors to the current study. In \cite{Wu}, we established the null controllability of the higher-dimensional equation \eqref{12.14.1} on a special control domain using a cut-off method. However, this special control domain must encompass the degenerate interior point $0\in\Omega$, which inherently limits its applicability. In \cite{Yang1}, we employed both the cut-off and approximation methods to derive the unique continuation property on an annular domain, marking the first attempt to apply the approximation method to degenerate parabolic equations.

In this paper, we have achieved the null controllability of \eqref{12.14.1} on arbitrary inner domains, also relying on the approximation method.
Degenerate PDEs have also garnered significant attention, with notable contributions including \cite{FF,CS2,CS3,Fabes,GC,Heinonen,Trudinger}. The solution spaces for degenerate PDEs are termed weighted Sobolev spaces (see \cite{GC,Heinonen}), which differ from the Sobolev spaces associated with uniformly elliptic PDEs. A natural approach to solving degenerate PDEs involves approximating them with a sequence of uniformly elliptic PDEs, i.e., employing the approximation method. While an effective approximation method utilizing the Calderon-Zygmund decomposition was introduced in \cite{FF,Cavalheiro}, the resulting approximate functions are generally only measurable and lack differentiability. To establish the unique continuation property or observability for degenerate PDEs, it is often necessary for the second-order or higher-order partial derivatives of the coefficients to satisfy certain conditions. Consequently, a novel approximation method must be devised, which is the primary motivation behind this work. Indeed, we previously utilized an approximation method in \cite{Wu1} to obtain the weak unique continuation property for the degenerate elliptic case of \eqref{12.18.1} on arbitrary inner domains. However, this method proved insufficient for the parabolic case, necessitating the development of a more robust approximation technique, which, while technically challenging, is feasible.

Carleman estimates \cite{Alabau,Araruna,Cannarsa,Cannarsa1,FG,Fursikov,Rousseau,Wu,Zuazua} constitute a powerful classical tool for establishing the unique continuation property or observability of PDEs. According to the Hilbert uniqueness method, controllability is equivalent to observability. For uniformly elliptic PDEs, there are at least two primary approaches to achieving controllability: Carleman estimates (\cite{Fursikov}) and the Lebeau-Robbiano spectral inequality (\cite{Rousseau}). In this work, we employ Carleman estimates to derive the observability of equation \eqref{12.14.1} using the approximation method. It is worth noting that the Green formula (or integration by parts) often fails to hold in the context of degenerate PDEs.  We shall demonstrate that the solution of \eqref{12.14.1} lacks sufficient regularity in higher-dimensional cases, rendering many computations based on the Green formula infeasible.  We also explain  that it appears impossible to directly derive Carleman estimates from equation \eqref{12.14.1}; instead, they must be obtained from the approximate uniformly elliptic PDEs. The aforementioned approximation method is suitable for degenerate interior points. For cases involving degenerate part boundaries, an alternative approximation method, known as the shape design method, can be employed, as discussed in \cite{Guo}.

 We proceed as follows. In Section \ref{S2}, we define the solution spaces for equation \eqref{12.14.1} and present a property of its solution. In Section \ref{S3}, we develop the approximation method. In Section \ref{S4}, we derive the Carleman estimate for the approximate uniformly parabolic equations. In Section \ref{S5}, we establish the observability for equation \eqref{12.14.1}.

\section{Solution spaces}\label{S2}

 In this section, we will delineate the solution space for equation \eqref{12.14.1} and elucidate a property of its solutions. To achieve this, we introduce the weighted Sobolev spaces (\cite{GC,Heinonen}).

\subsection{Solution spaces}

 When $p = 1$, a locally integrable non-negative function $w(\cdot)$ is said to be an $A_1$ weight if there exists a constant $c > 0$ such that for all cubes $K \subset \mathbb{R}^N$,
\begin{equation*}
\frac{1}{|K|}\int_K w(x) \, \mathrm{d}x \leq c \cdot \essinf_{K} w.
\end{equation*}
For $p \in (1, \infty)$, a locally integrable non-negative function $w(\cdot)$ is called an $A_p$ weight if there exists a constant $c > 0$ such that for all cubes $K$ in $\mathbb{R}^N$,
\begin{equation}\label{01.18.2}
\left(\frac{1}{|K|}\int_K w(x) \, \mathrm{d}x\right)\left(\frac{1}{|K|}\int_K w(x)^{-\frac{1}{p - 1}} \, \mathrm{d}x\right)^{p - 1} \leq c.
\end{equation}
The infimum of the set of constants $c = c(w, p) > 0$ for which \eqref{01.18.2} holds is referred to as the $A_p$ constant of $w(\cdot)$.
It is evident that $w(\cdot)$ is an $A_{1 + \frac{2}{N}}$ weight, given that $w(\cdot)$ is an $A_p$ weight if and only if $-N < \alpha < N(p - 1)$. Below, we define
\begin{equation*}
w(E) = \int_E w \, \mathrm{d}x.
\end{equation*}
A function $u(\cdot)$ belongs to $L^2(\Omega; w)$ if
\begin{equation*}
\int_\Omega u^2 w \, \mathrm{d}x < \infty.
\end{equation*}
The inner product on $L^2(\Omega; w)$ is given by
\begin{equation*}
(u, v)_{L^2(\Omega; w)} = \int_\Omega uvw \, \mathrm{d}x,
\end{equation*}
and the corresponding norm is
\begin{equation*}
\|u\|_{L^2(\Omega; w)} = \left(\int_\Omega u^2 w \, \mathrm{d}x\right)^{\frac{1}{2}}.
\end{equation*}
It is well-known that $(L^2(\Omega; w), (\cdot, \cdot){L^2(\Omega; w)})$ is a Hilbert space and $(L^2(\Omega; w), \|\cdot\|{L^2(\Omega; w)})$ is a Banach space.
We define
\begin{equation*}
H^{1}(\Omega; w) = \left\{u \in L^2(\Omega; w) \colon \frac{\partial u}{\partial x_i} \in L^2(\Omega; w), \, i = 1, \cdots, N\right\},
\end{equation*}
where $\frac{\partial u}{\partial x_i}, \, i = 1, \cdots, N$ are the distributional derivatives with respect to the spatial variables $x_1, \cdots, x_N$. The inner product on $H^1(\Omega; w)$ is
\begin{equation*}
(u, v)_{H^1(\Omega; w)} = \int_\Omega uvw \, \mathrm{d}x + \sum_{i = 1}^N \int_\Omega \frac{\partial u}{\partial x_i} \frac{\partial v}{\partial x_i} w \, \mathrm{d}x,
\end{equation*}
and the norm is
\begin{equation*}
\|u\|_{H^1(\Omega; w)} = \left(\int_\Omega u^2 w \, \mathrm{d}x + \sum_{i = 1}^N \int_\Omega \left|\frac{\partial u}{\partial x_i}\right|^2 w \, \mathrm{d}x\right)^{\frac{1}{2}}.
\end{equation*}
We define
\begin{equation*}
H_{0}^1(\Omega; w) = \text{the closure of } \mathcal{D}(\Omega) \text{ in } H^1(\Omega; w),
\end{equation*}
where $\mathcal{D}(\Omega) = C_0^\infty(\Omega)$ is the space of test functions.
We denote by $H^{-1}(\Omega; w)$ the dual space of $H_{0}^1(\Omega; w)$, which is a subspace of $\mathcal{D}'(\Omega)$. It is also well-known that $(H^1(\Omega; w), (\cdot, \cdot)_{H^1(\Omega; w)})$ is a Hilbert space and $(H^1(\Omega; w), \|\cdot\|_{H^1(\Omega; w)})$ is a Banach space.

 The subsequent Lemma \ref{08.16.L3} corresponds precisely to \cite[Theorem 1.3]{Fabes}.
\begin{lemma}\label{08.16.L3}
Let $\Omega \subset \mathbb{R}^N$ be an open bounded set. If $w \in A_p$ for $1 < p < \infty$, then there exist constants $C_\Omega$ and $\delta > 0$ such that for all $u \in C_0^\infty(\Omega)$ and all $k \in \left[1, \frac{N}{N - 1} + \delta\right]$,
\begin{equation*}
\|u\|_{L^{kp}(\Omega; w)} \leq C_\Omega \|\nabla u\|_{L^p(\Omega; w)},
\end{equation*}
where $C_\Omega$ depends only on $N$, the $A_p$ constant of $w(\cdot)$, $p$, and the diameter of $\Omega$.
\end{lemma}

The constant $C_\Omega$ in Lemma \ref{08.16.L3} is not explicitly given. In the following Lemma \ref{08.15.L1}, we provide a concrete form of $C_\Omega$, which is crucial for this work.
 \begin{lemma}\label{08.15.L1}
Let $N \geq 2$ and $\alpha \in (0, 2)$. Then, for all $u \in H_0^1(\Omega; w)$, the following inequality holds:
\begin{equation*}
(N - 2 + \alpha) \left\||x|^{\frac{\alpha}{2} - 1} u\right\|_{L^2(\Omega)} \leq 2 \|\nabla u\|_{L^2(\Omega; w)}.
\end{equation*}
Furthermore, if $u \in H_0^1(\Omega; w)$, it follows that $u \in L^2(\Omega)$.
\end{lemma}
\begin{proof}
The first part of the lemma can be derived from \cite[Lemma 3.1]{Wu} or \cite[Proposition 2.1 (1)]{Stuart}.
For the proof of the second part, from Lemma \ref{08.15.L1} and $\alpha \in (0, 2)$, we obtain
\begin{equation}\label{12.09.2}
\frac{N - 2 + \alpha}{2m} \|u\|_{L^2(\Omega; w)} \leq \|\nabla u\|_{L^2(\Omega; w)} \text{ with } m := \sup_{x \in \Omega} |x| + 1
\end{equation}
and
\begin{equation}\label{08.15.10}
\frac{N - 2 + \alpha}{2m^{1 - \frac{\alpha}{2}}} \|u\|_{L^2(\Omega)} \leq \|\nabla u\|_{L^2(\Omega; w)},
\end{equation}
which is the Poincar\'{e} inequality.
\end{proof}
\begin{remark}\label{08.16.R1}
If $u \in H_0^1(\Omega; w)$, from Lemma \ref{08.16.L3}, taking $k = 1$ and $p = 2$, we have
\begin{equation}\label{08.16.8}
\|u\|_{L^2(\Omega; w)} \leq C \|\nabla u\|_{L^2(\Omega; w)},
\end{equation}
where the constant $C > 0$ is independent of $u \in H_0^1(\Omega; w)$. This is the Poincar\'{e} inequality in the weighted Sobolev space.
In particular, the norm
\begin{equation}\label{08.19.2}
\|u\|_{H_0^1(\Omega; w)} = \left(\int_\Omega (\nabla u \cdot \nabla u) w \, \mathrm{d}x\right)^{\frac{1}{2}}
\end{equation}
is an equivalent norm in $H_0^1(\Omega; w)$, where $\nabla u = \left(\frac{\partial u}{\partial x_1}, \cdots, \frac{\partial u}{\partial x_N}\right)$ is the gradient of $u(\cdot)$. Hereafter, we use \eqref{08.19.2} to define the norm of $H_0^1(\Omega; w)$.
\end{remark}

The following Lemma \ref{08.15.L4} is \cite[Theorem 3.4]{Wu} or  \cite[Proposition 2.1 (5)]{Stuart}.
\begin{lemma}\label{08.15.L4}
The embedding $H_0^1(\Omega; w) \hookrightarrow L^2(\Omega)$ is compact.
\end{lemma}

We define
\begin{equation*}
L^2(0, T; H^1(\Omega; w)) = \left\{\varphi \in L^2(Q; w) \colon \frac{\partial \varphi}{\partial x_i} \in L^2(Q; w), \, i = 1, \cdots, N\right\},
\end{equation*}
where $Q = \Omega \times (0, T)$. Its inner product is
\begin{equation*}
(\varphi, \psi)_{L^2(0, T; H^1(\Omega; w))} = \iint_Q \varphi \psi w \, \mathrm{d}x \, \mathrm{d}t + \iint_Q (\nabla \varphi \cdot \nabla \psi) w \, \mathrm{d}x \, \mathrm{d}t,
\end{equation*}
and its norm is
\begin{equation*}
\|\varphi\|_{L^2(0, T; H^1(\Omega; w))}^2 = \iint_Q \varphi^2 w \, \mathrm{d}x \, \mathrm{d}t + \iint_Q |\nabla \varphi|^2 w \, \mathrm{d}x \, \mathrm{d}t.
\end{equation*}
Then $(L^2(0, T; H^1(\Omega; w)), (\cdot, \cdot){L^2(0, T; H^1(\Omega; w))})$ is a Hilbert space and $(L^2(0, T; H^1(\Omega; w)), \|\cdot\|{L^2(0, T; H^1(\Omega; w))})$ is a Banach space. The same conclusions hold for $L^2(0, T; H_0^1(\Omega; w))$.
By \eqref{08.16.8} in Remark \ref{08.16.R1}, the space $L^2(0, T; H_0^1(\Omega; w))$ has an equivalent norm
\begin{equation*}
\|\varphi\|_{L^2(0, T; H_0^1(\Omega; w))}^2 = \iint_Q |\nabla \varphi|^2 w \, \mathrm{d}x \, \mathrm{d}t.
\end{equation*}
We define
\begin{equation*}
W = \left\{\varphi \in L^2(0, T; H_0^1(\Omega; w)) \colon \partial_t \varphi \in L^2(0, T; H^{-1}(\Omega; w))\right\}.
\end{equation*}
It is clear that $W \subset C([0, T]; L^2(\Omega))$. Moreover, $W \hookrightarrow L^2(Q)$ is compact.
\begin{definition}\label{08.16.D1}
We call $\varphi \in W$ a \textit{weak solution} of \eqref{12.14.1} if
\begin{equation*}
-\iint_Q \varphi \partial_t \psi \, \mathrm{d}x \, \mathrm{d}t + \iint_Q (\nabla \varphi \cdot \nabla \psi) w \, \mathrm{d}x \, \mathrm{d}t = \iint_Q f \psi \, \mathrm{d}x \, \mathrm{d}t + \int_\Omega \varphi_0(x) \psi(x, 0) \, \mathrm{d}x
\end{equation*}
for any $\psi \in W$ with $\psi(T) = 0$.
\end{definition}
The following lemma is Lemma 3.7 of  \cite{FF}.
\begin{lemma}\label{12.09.L1}
Let $\varphi_0 \in L^2(\Omega)$, $S = \sum_{i = 1}^N \frac{\partial f_i}{\partial x_i} \in L^2(0, T; H_w^{-1}(\Omega))$, and $g \in L^2(Q; w^{-1})$. If $\varphi \in L^2(0, T; H_0^1(\Omega; w))$ is a weak solution of the problem
\begin{equation}\label{06.07.2}
\begin{cases}
\partial_t \varphi - \mathcal{A} \varphi = g - S, & \text{in } Q, \\
\varphi = 0, & \text{on } \partial Q, \\
\varphi(0) = \varphi_0, & \text{in } \Omega,
\end{cases}
\end{equation}
then
\begin{equation}\label{12.09.1}
\begin{split}
&\sup_{t \in [0, T]} \int_\Omega |\varphi(x, t)|^2 \, \mathrm{d}x + \iint_Q |\nabla \varphi(x, t)|^2 w \, \mathrm{d}x \, \mathrm{d}t \\
&\leq C \left(\|\varphi_0\|_{L^2(\Omega)}^2 + \|g\|_{L^2(Q; w^{-1})}^2 + \sum_{j = 1}^N \|f_j\|_{L^2(\Omega; w^{-1})}^2\right),
\end{split}
\end{equation}
where the constant $C > 0$ depends only on $\alpha$, $N$, and $\Omega$.
Moreover, if $g \in L^2(Q)$ and $S = 0$, then there exists a unique solution $u \in L^2(0, T; H_0^1(\Omega; w))$ of \eqref{06.07.2}, and
\begin{equation}\label{06.07.4}
\begin{split}
\sup_{t \in [0, T]} \int_\Omega |\varphi(x, t)|^2 \, \mathrm{d}x + \iint_Q |\nabla \varphi(x, t)|^2 w \, \mathrm{d}x \, \mathrm{d}t
&\leq C \left(\|\varphi_0\|_{L^2(\Omega)}^2 + \|g\|_{L^2(Q)}^2\right).
\end{split}
\end{equation}
\end{lemma}
\begin{proof}
 It is easily verified that
\begin{equation}\label{06.07.3}
\begin{split}
&\iint_{\Omega \times (0, \tau)} w \nabla \varphi \cdot \nabla \varphi \, \mathrm{d}x \, \mathrm{d}t + \frac{1}{2} \int_\Omega |\varphi(x, \tau)|^2 \, \mathrm{d}x \\
&= \frac{1}{2} \int_\Omega |\varphi_0|^2 \, \mathrm{d}x + \iint_{\Omega \times (0, \tau)} g \varphi \, \mathrm{d}x \, \mathrm{d}t + \iint_{\Omega \times (0, \tau)} \sum_{i = 1}^N \frac{\partial \varphi}{\partial x_i} f_i \, \mathrm{d}x \, \mathrm{d}t \\
&\leq \frac{1}{2} \int_\Omega |\varphi_0|^2 \, \mathrm{d}x + \frac{1}{2\varepsilon} \iint_{\Omega \times (0, \tau)} g^2 w^{-1} \, \mathrm{d}x \, \mathrm{d}t + \varepsilon \iint_{\Omega \times (0, \tau)} \varphi^2 w \, \mathrm{d}x \, \mathrm{d}t \\
&\quad + \frac{1}{2\varepsilon} \iint_{\Omega \times (0, \tau)} \sum_{i = 1}^N f_i^2 w^{-1} \, \mathrm{d}x \, \mathrm{d}t + \varepsilon \iint_{\Omega \times (0, \tau)} |\nabla \varphi|^2 w \, \mathrm{d}x \, \mathrm{d}t
\end{split}
\end{equation}
for every $\tau \in (0, T)$ and for any $\varepsilon > 0$. Taking $\varepsilon = \frac{1}{4} \min\left\{\left(\frac{N - 2 + \alpha}{2m}\right)^2, 1\right\}$, by \eqref{12.09.2}, we obtain \eqref{12.09.1}.
Finally, if $g \in L^2(\Omega)$ and $S = 0$, we have
\begin{equation*}
\begin{split}
2 \iint_{\Omega \times (0, \tau)} g \varphi \, \mathrm{d}x \, \mathrm{d}t
&\leq \iint_{\Omega \times (0, \tau)} g^2 \, \mathrm{d}x \, \mathrm{d}t + \iint_{\Omega \times (0, \tau)} \varphi^2 \, \mathrm{d}x \, \mathrm{d}t \\
&\leq \iint_{\Omega \times (0, \tau)} g^2 \, \mathrm{d}x \, \mathrm{d}t + C \iint_{\Omega \times (0, \tau)} |\nabla \varphi|^2 w \, \mathrm{d}x \, \mathrm{d}t
\end{split}
\end{equation*}
by Lemma \ref{08.15.L1}.
We then use this inequality in the first inequality in \eqref{06.07.3} to obtain \eqref{06.07.4}.
\end{proof}

From Lemma \ref{12.09.L1}, we obtain the following Corollary \ref{12.11.C1}.
\begin{corollary}\label{12.11.C1}
Under the assumptions in Lemma \ref{12.09.L1} with $g = f_i = 0, \, i = 1, \cdots, N$, there holds
\begin{equation*}
\int_\Omega |\varphi(x, \tau)|^2 \, \mathrm{d}x \leq \int_\Omega |\varphi(x, t)|^2 \, \mathrm{d}x
\end{equation*}
for all $0 \leq t \leq \tau \leq T$.
\end{corollary}

\subsection{A property of solution of \eqref{12.14.1}}\label{S2.2}

 In this section, we present a property of the solution to \eqref{12.14.1} (Lemma \ref{05.02.L1}). This property indicates that the integration by parts formula
\begin{equation}\label{05.08.3}
\iint_Q \Div(A \nabla u)(A\nabla \phi\cdot \nabla u)\df x=\iint_\Om A \nabla u \cdot\nabla (A\nabla\phi\cdot\nabla u)\df x+\iint_{\pt\Om} (A \nabla u\cdot \nu) (A\nabla\phi\cdot \nabla u)\df x
\end{equation}
fails to hold when $\phi = |x|^{2-\al}$. It is worth noting that the equality \eqref{05.08.3} plays a crucial role in the Carleman estimate (Section \ref{S4.1}). The derivation of this property involves the quotient difference (\cite[Chapter 5.8.2, p. 277]{Evans})
\begin{equation*}
\pt_k^hu(x)=\f{u(x+h e_k)-u(x)}{h}, \quad x\in\R^N, \ 0\neq h\in\R, \ k\in J_N,
\end{equation*}
where $\{e_k=(0,\cdots, 0,1,0,\cdots,0)\}_{k\in\N}$ is the standard basis of $\R^N$. We also define
\begin{equation*}
u^h(x)=u(x+h) \quad \text{for } x\in \R^N, \ 0\neq h\in\R.
\end{equation*}
For functions $v,\xi$, and indices $k,i\in J_N$ and $0\neq h\in\R$, the following identities hold:
\begin{equation*}
\begin{split}
\int_\Om v\pt_k^{-h}\xi \df x=-\int_\Om \xi \pt_k^h v\df x, \quad \pt_i\pt_k^h v=\pt_k^h\pt_iv,
\end{split}
\end{equation*}
and
\begin{equation*}
\pt_k^h(v\xi)=v^h \pt_k^h \xi +(\pt_k^h v)\xi.
\end{equation*}
We denote $\nabla=(\pt_1,\cdots, \pt_N)=(\pt_{x_1},\cdots, \pt_{x_N})$.
\begin{lemma}\label{05.04.L1}
Let $u\in H_0^1(\Om;w)$. Then for any $U\Subset \Om$, there exists a constant $C>0$ such that for all $0<|h|<\f{1}{2}\dist(U,\pt\Om)$,
\begin{equation*}
\int_U (\pt_k^hu)^2 \left(|x|^{\al+2}\right)^h\df x\leq C\int_\Om |\pt_k u|^2w\df x.
\end{equation*}
\end{lemma}
\begin{proof}
The proof deviates from the classical case due to the non-transform invariance of the measure $w\df x$. We assume $u\in C^\iy(\ol\Om)$.
To establish $(|x|^{\f{\al}{2}+1})^h\pt_k^hu\in L^2(U)$ for any $k\in J_N$, we observe that
\begin{equation*}
\pt_k^h \left(|x|^{\f{\al}{2}+1}u\right)=\left(|x|^{\f{\al}{2}+1}\right)^h \pt_k^h u+\left(\pt_k^h|x|^{\f{\al}{2}+1}\right)u.
\end{equation*}
Thus, it suffices to show that
\begin{equation}\label{05.04.1}
\pt_k^h\left(|x|^{\f{\al}{2}+1}u\right), \left(\pt_k^h|x|^{\f{\al}{2}+1}\right)u\in L^2(\Om).
\end{equation}
First, we prove
\begin{equation}\label{05.04.2}
\left|\pt_k^h |x|^{\f{\al}{2}+1}\right|\leq C\left(|x|^{\f{\al}{2}}+|x+he_k|^{\f{\al}{2}}\right),
\end{equation}
where the constant $C>0$ depends only on $\al$ and $m$.
We consider the one-dimensional case $g(x)=|x+a|^{\f{\al}{2}+1}$ for $x\in (-m,m)$ and $a\in\R$. Since $g(x)$ ($0<\al<2$) is convex for $x\in (0,m)$ as
\begin{equation*}
g''(x)=\f{\al}{2}\left(\f{\al}{2}+1\right)|x+a|^{\f{\al}{2}-1}\geq 0,
\end{equation*}
we have
\begin{equation*}
\left|\pt_k^h|x|^{\f{\al}{2}+1}\right|\leq \f{\al}{2} \left(\f{\al}{2}+1\right)\max\{|x|^{\f{\al}{2}}, |x+h|^{\f{\al}{2}}\}.
\end{equation*}
This implies \eqref{05.04.2}.
From \eqref{05.04.2}, we obtain
\begin{equation}\label{06.07.5}
\begin{split}
\int_\Om \left|\left(|x|^{\f{\al}{2}+1}\right)^hu\right|^2\df x
&\leq C\left(\int_\Om |x|^\al u^2\df x+\int_\Om |x+he_k|^{\al}u^2\df x\right)<+\iy,
\end{split}
\end{equation}
where the constant $C>0$ depends only on $\al$ and $m$. This proves the second part of \eqref{05.04.1}.
For any $x\in U, k\in J_N$, and $0<|h|<\dist(U,\pt\Om)$, we have
\begin{equation*}
 |x+h|^{\f{\al}{2}+1}u(x+he_k)-|x|^{\f{\al}{2}+1}u(x)\leq h\int_0^1 \left[\pt_k\left(|x|^{\f{\al}{2}+1}u\right)\right](x+ t h e_k)\df t
\end{equation*}
due to $|x|^{\f{\al}{2}+1}\in C^1(\ol\Om)$ and $u\in C^\iy(\ol\Om)$. Consequently,
\begin{equation*}
\begin{split}
\int_U \left|\pt_k^h \left(|x|^{\f{\al}{2}+1}u\right)\right|^2\df x
&\leq \int_U \int_0^1 \left[\left(\f{\al}{2}+1\right) |x|^{\f{\al}{2}-1}x_k u+|x|^{\f{\al}{2}+1}\pt_k u\right]^2(x+
t h e_k)\df t\df x\\
&\leq C\int_\Om \left(u^2+|x|^\al |\pt_k u|^2\right)\df x,
\end{split}
\end{equation*}
where the constants $C>0$ depend only on $\al$ and $m$.  This proves the first part of \eqref{05.04.1}.
\end{proof}
\begin{remark}
From the proof of Lemma \ref{05.04.L1}, we cannot derive the following classical results:
\begin{equation*}
\int_U (\pt_k^hu)^2w\df x\leq C\int_\Om |\nabla u|^2w\df x \text{ or }   \int_U (\pt_k^hu)^2w^h\df x\leq C\int_\Om |\nabla u|^2w\df x.
\end{equation*}
This issue arises from the non-transform invariance of the measure $w\df x$ and the fact that $w=|x|^\al\notin C^1(\Om)$ (see \eqref{06.07.5} and the non-meaningfulness of $\int_\Om |x+he_k|^{\al-2}u^2\df x$ for $0\neq h\in\R$).
\end{remark}
\begin{lemma}\label{05.02.L1}
Let $u(\cdot)$ be a solution of equation \eqref{12.14.1}. Then for any $U\Subset \Om$,
\begin{equation*}
|x|^{3+\f{3\al}{2}}\f{\pt^2u}{\pt x_j\pt x_k}\in  L^2(U\ts (0,T)), \quad \forall j,k\in J_N.
\end{equation*}
\end{lemma}
\begin{proof}
This is essentially a problem of elliptic regularity. We only need to prove the case where $u\in H^1(\Om;w)$ is a weak solution of the following system
\begin{equation}\label{06.07.1}
\begin{cases}
-\Div(w\nabla u)=g \quad \text{in } \Om,\\
g=\chi_\om f-\pt_t u\in L^2(\Om),
\end{cases}
\end{equation}
and then
\begin{equation}\label{06.07.6}
|x|^{3+\f{3\al}{2}}\f{\pt^2u}{\pt x_j\pt x_k}\in L^2(U), \quad \forall j,k\in J_N.
\end{equation}
The partial differential operator $-\Div(|x|^\al \nabla\bullet)$ has a discrete point spectrum
\begin{equation*}
0<\la_1<\la_2\leq \la_3\leq \cdots\ra+\iy
\end{equation*}
from Lemma \ref{08.15.L4}. Then there exists an orthonormal basis $\{\Phi_n\}_{n\in\N}\s L^2(\Om)$
(\cite[Theorem 7, Appendix E, p.645]{Evans}), and moreover, $\{\la_n^\f{1}{2}\Phi_n\}_{n\in\N}$ is an orthonormal basis of $H_0^1(\Om;w)$, where $\Phi_n$ is the eigenfunction of $-\Div(|x|^\al \nabla\bullet)$ corresponding to the eigenvalue $\la_n\ (n\in\N)$. We denote $g=\sum_{n\in\N}g_n\Phi_n$ with $\{g_n\}_{n\in\N}\s\R$, then $u=\sum_{n\in\N}g_n\la_n^{-1}\Phi_n$ is the solution of \eqref{06.07.1}. Hence, to show \eqref{06.07.6}, it suffices to show
\begin{equation*}
|x|^{3+\f{3\al}{2}}\f{\pt^2\Phi_n}{\pt x_j\pt x_k}\in L^2(U), \quad \forall j,k\in J_N.
\end{equation*}

We split the proof into two steps.

{\it Step 1. } Let $W$ be an open set satisfying $U\Subset W\Subset \Om$. Take $\zeta\in C^\iy(\Om), 0\leq \zeta\leq 1$ and
\begin{equation*}
\zeta=1 \quad \text{on }U,\quad \zeta=0  \text{ on }\R^N\se W.
\end{equation*}
Let $0<|h|<\f{1}{2}\min\{\dist(U,\pt W), \dist(W, \pt U)\}$. For any $k\in J_N$, choose
\begin{equation*}
v=-\pt_k^{-h}\left[\zeta^2|x|^{2+\al}(|x|^{4+\al})^h\pt_k^hu\right],
\end{equation*}
where $\pt_k^hu$ is the difference quotient.
It is evident that $v\in H_0^1(\Om;w)$. Multiplying $v$ on both sides of the following equation
\begin{equation}\label{06.07.7}
\begin{cases}
-\Div(w\nabla \Phi_n)=\la_n\Phi_n, &\text{in } \Om, \\
\Phi_n=0, &\text{on }\pt\Om,
\end{cases}
\end{equation}
we obtain
\begin{equation}\label{05.02.2}
\int_\Om |x|^\al \nabla \Phi_n\cdot \nabla v\df x=\la_n\int_\Om \Phi_n v\df x.
\end{equation}
We have
\begin{equation*}
\begin{split}
&\int_\Om |x|^\al \nabla \Phi_n\cdot\nabla v\df x\\
&=\sum_{i=1}^N\int_\Om \pt_k^h\left(|x|^\al \pt_i\Phi_n\right)\pt_i\left[\zeta^2|x|^{2+\al}(|x|^{4+\al})^h\pt_k^h\Phi_n\right]\df x\\
&=\sum_{i=1}^N\int_\Om \left[(|x|^\al)^h \pt_k^h\pt_i\Phi_n+(\pt_k^h|x|^\al)\pt_i\Phi_n\right]\\
&\hspace{14.5mm}\ts \bigg[2\zeta(\pt_i\zeta)|x|^{2+\al}(|x|^{4+\al})^h\pt_k^h\Phi_n+(2+\al)\zeta^2|x|^{\al}x_i(|x|^{4+\al})^h\pt_k^h\Phi_n\\
&\hspace{20mm}+(4+\al)\zeta^2|x|^{2+\al}(|x+h|^{2+\al}(x+h)_i)\pt_k^h\Phi_n+\zeta^2|x|^{2+\al}(|x|^{4+\al})^h \pt_k^h\pt_i\Phi_n\bigg]\df x\\
&=\sum_{i=1}^N\int_\Om \zeta^2|x|^{2+\al}(|x|^{4+2\al})^h (\pt_k^h\pt_i\Phi_n)^2\df x+A_1+A_2,
\end{split}
\end{equation*}
where $(x+h)_i$ is the $i$-th component of $x+h$, with
\begin{equation*}
\begin{split}
A_1
&=\sum_{i=1}^N\int_\Om (|x|^\al)^h\pt_k^h\pt_i\Phi_n\bigg[2\zeta(\pt_i\zeta)|x|^{2+\al}(|x|^{4+\al})^h\pt_k^h\Phi_n+(2+\al)\zeta^2|x|^{\al}x_i(|x|^{4+\al})^h\pt_k^h\Phi_n\\
&\hspace{40mm}+(4+\al)\zeta^2|x|^{2+\al}\left(|x+h|^{2+\al}(x+h)_i\right)\pt_k^h\Phi_n\bigg]\df x,
\end{split}
\end{equation*}
and
\begin{equation*}
\begin{split}
A_2
&=\sum_{i=1}^N\int_\Om (\pt_k^h|x|^\al)\pt_i\Phi_n\bigg[2\zeta(\pt_i\zeta)|x|^{2+\al}(|x|^{4+\al})^h\pt_k^h\Phi_n+(2+\al)\zeta^2|x|^{\al}x_i(|x|^{4+\al})^h\pt_k^h\Phi_n\\
&\hspace{20mm}+(4+\al)\zeta^2|x|^{2+\al}\left(|x+h|^{2+\al}(x+h)_i\right)\pt_k^h\Phi_n+\zeta^2|x|^{2+\al}(|x|^{4+\al})^h \pt_k^h\pt_i\Phi_n\bigg]\df x.
\end{split}
\end{equation*}

Next, we estimate the terms $A_1$ and $A_2$. We have
\begin{equation*}
\begin{split}
A_1
&\leq \f{1}{4}\sum_{i=1}^N\int_\Om \zeta^2|x|^{2+\al}(|x|^{4+2\al})^h (\pt_k^h\pt_i\Phi_n)^2\df x\\
&\hspace{4.5mm}+16\sum_{i=1}^N\int_\Om(\pt_i\zeta)^2|x|^{2+\al}(|x|^{4+2\al})^h(\pt_k^h\Phi_n)^2\df x+8(2+\al)\int_\Om \zeta^2|x|^{\al}(|x|^{4+2\al})^h(\pt_k^h\Phi_n)^2\df x\\
&\hspace{4.5mm}+8(4+\al)\int_\Om \zeta^2|x|^{2+\al}(|x|^{2+2\al})^h(\pt_k^h\Phi_n)^2\df x,
\end{split}
\end{equation*}
and
\begin{equation*}
\begin{split}
A_2
&\leq\f{1}{4}\sum_{i=1}^N\int_\Om \zeta^2|x|^{2+\al}(|x|^{4+2\al})^h (\pt_k^h\pt_i\Phi_n)^2\df x+2\sum_{i=1}^N\int_\Om \zeta^2|x|^{2+\al}(|x|^{4})^h(\pt_k^h|x|^\al)^2(\pt_i\Phi_n)^2\df x\\
&\hspace{4.5mm}+\int_\Om \zeta^2|x|^\al|\nabla \Phi_n|^2\df x+\sum_{i=1}^N\int_\Om (\pt_i\zeta)^2|x|^{4+\al}(|x|^{8+2\al})^h(\pt_k^h|x|^\al)^2(\pt_k^h\Phi_n)^2\df x\\
&\hspace{4.5mm}+\int_\Om \zeta^2|x|^\al|\nabla \Phi_n|^2\df x+\f{2+\al}{4}\sum_{i=1}^N\int_\Om \zeta^2|x|^{2+\al}(|x|^{8+2\al})^h(\pt_k^h|x|^\al)^2(\pt_k^h\Phi_n)^2\df x\\
&\hspace{4.5mm}+\int_\Om \zeta^2|x|^\al |\nabla \Phi_n|^2\df x+\f{2+\al}{4}\sum_{i=1}^N\int_\Om \zeta^2|x|^{4+\al}(|x|^{6+2\al})^h (\pt_k^h|x|^\al)^2(\pt_k^h\Phi_n)^2\df x.
\end{split}
\end{equation*}
Similar to \eqref{05.04.2}, we obtain
\begin{equation*}
|x|^{1+\f{\al}{2}}\pt_k^h|x|^\al =\pt_k^h|x|^{\al+1+\f{\al}{2}}-(|x|^\al)^h\pt_k^h|x|^{1+\f{\al}{2}}
\end{equation*}
is a bounded function on $\Om$.
Then, from Lemma \ref{05.04.L1},
we have
\begin{equation*}
\begin{split}
|A_1|, |A_2|
&\leq \f{1}{4}\sum_{i=1}^N\int_\Om \zeta^2|x|^{2+\al}(|x|^{4+2\al})^h (\pt_k^h\pt_i\Phi_n)^2\df x+C\|\Phi_n\|_{H^1(\Om;w)}^2,
\end{split}
\end{equation*}
where the constants $C>0$ depend only on $\al$ and $m$. These show that
\begin{equation}\label{05.04.3}
\int_\Om |x|^\al \nabla \Phi_n\cdot \nabla v\df x\geq \f{1}{2}\sum_{i=1}^N\int_\Om \zeta^2|x|^{2+\al}(|x|^{4+2\al})^h (\pt_k^h\pt_i\Phi_n)^2\df x-C\|\Phi_n\|_{H^1(\Om;w)}^2,
\end{equation}
where the constants $C>0$ depend only on $\al$ and $m$.

{\it Step 2.}  Note that
\begin{equation*}
\begin{split}
\la_n\int_\Om \Phi_nv\df x
&=\la_n\int_\Om \zeta^2|x|^{2+\al}(|x|^{4+\al})^h(\pt_k^h\Phi_n)^2\df x\leq C\la_n \|\Phi_n\|_{H^1(\Om;w)}^2
\end{split}
\end{equation*}
from Lemma \ref{05.04.L1}, where the constant $C>0$ depends only on $\al$ and $m$. This, together with \eqref{05.04.3}, yields
\begin{equation}\label{05.04.4}
\sum_{i=1}^N\int_\Om \zeta^2|x|^{2+\al}(|x|^{4+2\al})^h (\pt_k^h\pt_i\Phi_n)^2\df x\leq C\la_n\|\Phi_n\|_{H^1(\Om;w)}^2,
\end{equation}
where the constant $C>0$ depends only on $\al$ and $m$.

Finally, since
\begin{equation*}
(|x|^{1+\f{\al}{2}})^h\pt_k^h\pt_i\Phi_n=\pt_k^h(|x|^{1+\f{\al}{2}}\pt_i\Phi_n)-(\pt_k^h|x|^{1+\f{\al}{2}})(\pt_i\Phi_n),
\end{equation*}
and noting that $|x|^{1+\f{\al}{2}}\pt_i\Phi_n\in L^2(\Om)$, we have
\begin{equation*}
\begin{split}
&\int_U |x|^{2+\al} (|x|^{2+\al})^h \left[\pt_k^h(|x|^{1+\f{\al}{2}}\pt_i\Phi_n)\right]^2\df x\\
&=\int_U|x|^{2+\al}(|x|^{4+2\al})^h(\pt_k^h\pt_i\Phi_n)^2\df x+\int_U |x|^{2+\al}(|x|^{2+\al})^h \left(\pt_k^h |x|^{1+\f{\al}{2}}\right)^2(\pt_i\Phi_n)^2\df x<+\iy
\end{split}
\end{equation*}
by Lemma \ref{05.04.L1}. That is, $|x|^{1+\f{\al}{2}}(|x|^{1+\f{\al}{2}})^h \pt_k^h(|x|^{1+\f{\al}{2}}\pt_i\Phi_n)\in L^2(U)$ for all $0<|h|<\f{1}{2}\dist (U,\pt\Om)$. Hence, there exists a function $v_k\in L^2(U)$ and a subsequence $h_l\ra 0$ such that
\begin{equation*}
|x|^{1+\f{\al}{2}}(|x|^{1+\f{\al}{2}})^{h_l} \pt_k^h (|x|^{1+\f{\al}{2}}\pt_i\Phi_n)\ra v_k   \text{ in } L^2(U).
\end{equation*}
For any $\psi\in \mcD(U)$, we have
\begin{equation*}
\begin{split}
\int_U v_k\psi\df x
&=\lim_{h_l\ra 0}\int_U \left[|x|^{1+\f{\al}{2}}(|x|^{1+\f{\al}{2}})^{h_l} \pt_k^h (|x|^{1+\f{\al}{2}}\pt_i\Phi_n)\right] \psi\df x\\
&=\lim_{h_l\ra 0}\int_U \left[\pt_k^h(|x|^{1+\f{\al}{2}}\pt_i\Phi_n)\right] \left[|x|^{1+\f{\al}{2}}(|x|^{1+\f{\al}{2}})^{h_l}\psi\right]\df x\\
&=-\lim_{h_l\ra 0}\int_U \left(|x|^{1+\f{\al}{2}}\pt_i\Phi_n\right)\pt_k^{-h}\left[|x|^{1+\f{\al}{2}}(|x|^{1+\f{\al}{2}})^{h_l}\psi\right]\df x\\
&=-\int_U|x|^{1+\f{\al}{2}}(\pt_i\Phi_n) \pt_k \left(|x|^{2+\al}\psi\right)\df x\\
&=-(2+\al)\int_U |x|^{1+\f{3\al}{2}}x_k (\pt_i\Phi_n)\psi \df x-\int_U |x|^{3+\f{3\al}{2}}(\pt_i\Phi_n)\pt_k\psi\df x,
\end{split}
\end{equation*}
which implies that
\begin{equation*}
\pt_k \left(|x|^{3+\f{3\al}{2}}(\pt_i\Phi_n)\right)=v_k+(2+\al)|x|^{1+\f{3\al}{2}}x_k\pt_i\Phi_n\in L^2(U).
\end{equation*}
This completes the proof of the lemma.
\end{proof}
\begin{remark}\label{05.16.R1}
From Lemma \ref{05.02.L1}, we have $|x|^{3+\f{3\al}{2}}\pt_{x_ix_j}u\in L^2(\Om)$ for $i,j\in J_N$. However, it is challenging to obtain $|x|^\be \pt_{x_ix_j}u\in L^2(\Om)$ for $\be<3+\f{3\al}{2}$. This makes the first step \eqref{05.08.3} of the Carleman estimate (or, the subsequent computation and estimation) impossible. This is the distinction between the higher-dimensional case and the one-dimensional case. This is why we employ the approximation method in the following.
There is limited literature on the properties of the second derivatives of degenerate partial differential equations. However, for unique continuation properties or observability inequalities, the existence of some form of second or higher derivatives of the degenerate partial differential equations is a fundamental necessity.
\end{remark}

\section{Approximations}\label{S3}

 In this section, we employ uniformly parabolic equations to approximate the equation \eqref{12.14.1}. To this end, we introduce weight functions to approximate the $A_{1+\frac{2}{N}}$ weight function $w(\cdot)$ (Lemma \ref{12.17.L1}). Subsequently, we present one of the main results of this paper, Theorem \ref{08.16.T1}.
\begin{lemma}\label{12.17.L1}
There exists a $C^{2,1}$ function $\psi_\e: \mathbb{R}\to\mathbb{R}$ with $\e\in (0,\frac{1}{2})$ such that
\begin{equation*}
\psi_\e(x)=|x| \quad \text{for } |x|\geq \e, \quad |x|\leq \psi_\e(x)\leq 2\e \quad \text{on } [-\e, \e],
\end{equation*}
and
\begin{equation}\label{12.16.1}
\psi_\e\geq \frac{\e}{4} \quad \text{on } \mathbb{R}, \quad |\psi_\e'|\leq C   \text{ on } \mathbb{R}, \quad |\psi_\e''|\leq \frac{C}{\e},  \text{ and } |\psi_\e'''|\leq \frac{C}{\e^2}  \text{ on } \mathbb{R},
\end{equation}
where the constants $C>0$ are absolute.
\end{lemma}
\begin{proof}
We consider the following polynomial function
\begin{equation*}
\psi_\e(x)=\sum_{i=0}^4a_ix^{4-i},
\end{equation*}
where $a_0, \cdots, a_4$ are constants to be determined. We impose the conditions
\begin{equation*}
\begin{split}
\psi_\e(-\e)=\e, \quad \psi_\e(\e)=\e, \quad \psi_\e'(-\e)=-1, \quad \psi_\e'(\e)=1, \quad \psi_\e''(\pm \e)=0,
\end{split}
\end{equation*}
which yield the solution
\begin{equation}\label{12.20.1}
\psi_\e(x)=\frac{3\e}{8}+\frac{3}{4\e}x^2-\frac{1}{8\e^3}x^4, \quad x\in[-\e,\e].
\end{equation}
This function satisfies the required properties.
\end{proof}
\begin{remark}\label{06.07.R1}
Let $x\in\mathbb{R}^N$, and consider the function $\psi_\e(x)=\psi_\e(|x|)$ defined in Lemma \ref{12.17.L1}. We have
\begin{equation}\label{12.18.1}
\begin{split}
\nabla\psi_\e
&=\frac{3}{2\e}x-\frac{1}{2\e^3}|x|^2x, \quad
\frac{\partial^2 \psi_\e}{\partial x_i\partial x_j}
=\delta_{ij}\left[\frac{3}{2\e}-\frac{1}{2\e^3}|x|^2\right]-\frac{1}{\e^3}x_ix_j, \\
\Delta\psi_\e
&=N\left[\frac{3}{2\e}-\frac{1}{2\e^3}|x|^2\right]-\frac{1}{\e^3}|x|^2, \\
(D^2\psi_\e)\nabla\psi_\e
&=\left(\frac{3}{2\e}-\frac{1}{2\e^3}|x|^2\right)^2x-\frac{1}{\e^3}\left(\frac{3}{2\e}-\frac{1}{2\e^3}|x|^2\right)|x|^2x
\end{split}
\end{equation}
on $B_\e$,  and
\begin{equation}\label{12.18.2}
\begin{split}
\nabla \psi_\e=\frac{x}{|x|}, \quad \frac{\partial^2\psi_\e}{\partial x_i\partial x_j}=\delta_{ij}|x|^{-1}-x_ix_j|x|^{-3}, \quad \Delta\psi_\e=(N-1)|x|^{-1}, \quad (D^2\psi_\e)\nabla\psi_\e=0
\end{split}
\end{equation}
on $\Omega\setminus B_\e$.
\end{remark}
Denote
\begin{equation*}
\begin{split}
w_\e(x)= (\psi_\e(x))^\al=(\psi_\e(|x|))^\al, \quad x\in\Omega.
\end{split}
\end{equation*}
\begin{lemma}\label{01.18.L1}
Let $0<\e\ll 1$. Then the function $w_\e(\cdot)$ is an $A_{1+\frac{2}{N}}$ weight. Moreover, there exists a constant $C>0$ independent of $\e$ such that the $A_p$ constant $c(w_\e, 1+\frac{2}{N})\leq C$.
\end{lemma}
\begin{proof}
It is evident that $w_\e(\cdot)$ is an $A_1$ weight since
\begin{equation*}
\frac{1}{|K|}\int_Kw(x)\df x\leq \frac{4m}{\e} \essinf_K w,
\end{equation*}
and $\essinf_Kw_\e\geq \frac{\e}{4}$. By the property $A_{p_1}\subset A_{p_2}$ for $1\leq p_1<p_2$, we conclude that $w_\e(\cdot)$ is an $A_{1+\frac{2}{N}}$ weight.
Next, we show that $\{c(w_\e, 1+\frac{2}{N})\colon \e\in (0,1)\}$ is uniformly bounded above, where $c(w_\e,1+\frac{2}{N})$ is the $A_{1+\frac{2}{N}}$ constant of $w_\e(\cdot)$.
Firstly, for $y=(y_1,\cdots ,y_N)\in \Omega$, consider the case $K=\prod_{i=1}^N(y_i-a,y_i+a)$ for $0<a\ll \e$. We have
\begin{equation*}
\begin{split}
\frac{1}{|K|}\int_K w_\e(x)\df x
&\leq 2 (2a)^{-N}\iint_{B_{\sqrt{N}a}}w_\e(y) \df r=2(2a)^{-N}w_\e(y)|B_{\sqrt{N}a}|=\frac{| B_1|N^\frac{N}{2}}{2^{N-1}}  w_\e(y),
\end{split}
\end{equation*}
and
\begin{equation*}
\begin{split}
\left(\frac{1}{|K|}\int_Kw_\e(x)^{-\frac{N}{2}}\df x\right)^\frac{2}{N}
&\leq  2w_\e(y)^{-1}\left(\frac{| B_1|N^\frac{N}{2}}{2^N}\right)^\frac{2}{N}=w_\e(y)^{-1}\frac{| B_1|^\frac{2}{N}N}{2}.
\end{split}
\end{equation*}
These imply that
\begin{equation*}
\begin{split}
\left(\frac{1}{|K|}\int_K w_\e(x)\df x\right)\left(\frac{1}{|K|}\int_Kw_\e(x)^{-\frac{N}{2}}\df x\right)^\frac{2}{N}\leq \frac{1}{2^N}| B_1|^{1+\frac{2}{N}}N^{1+\frac{N}{2}}.
\end{split}
\end{equation*}
Secondly, for $y=(y_1,\cdots, y_N)\in\Omega$, consider the case $K_y=\prod_{i=1}^N(y_i-a,y_i+a)$ for $a\approx \e$. We have
\begin{equation*}
\begin{split}
&\left(\frac{1}{|K_y|}\int_{K_y} w_\e(x)\df x\right)\left(\frac{1}{|K_y|}\int_{K_y}w_\e(x)^{-\frac{N}{2}}\df x\right)^\frac{2}{N}\\
&\leq (3\e)^\al \frac{1}{4\e^2}\left(\int_{B_{\sqrt{N}\e}}|x|^{-\frac{N\al}{2}}\df x\right)^\frac{2}{N}=(3\e)^\al \frac{1}{4\e^2} |\partial B_1|^\frac{2}{N}\left(\int_0^{\sqrt{N}\e} r^{N-1-\frac{N\al}{2}}\df r\right)^\frac{2}{N}\\
&\leq \frac{3^\al}{2^{2+\frac{2}{N}}}|\partial B_1|^\frac{2}{N} (2-\al)^{-\frac{2}{N}}N^{2-\al-\frac{2}{N}}
\end{split}
\end{equation*}
for $|y|\leq 2\e$. The other cases are analogous to the case of $|x|^\al$.
Thirdly, for $y=(y_1,\cdots, y_N)\in\Omega$, consider the case  $K_y=\prod_{i=1}^N(y_i-a,y_i+a)$ for $\e\ll a$. Then,
\begin{equation*}
\begin{split}
&\left(\frac{1}{|K_y|}\int_{K_y} w_\e(x)\df x\right)\left(\frac{1}{|K_y|}\int_{K_y}w_\e(x)^{-\frac{N}{2}}\df x\right)^\frac{2}{N}\\
&\approx \left(\frac{1}{|K_y|}\int_{K_y} w(x)\df x\right)\left(\frac{1}{|K_y|}\int_{K_y}w(x)^{-\frac{N}{2}}\df x\right)^\frac{2}{N}\leq c\left(w,1+\frac{2}{N}\right).
\end{split}
\end{equation*}
Overall, we obtain that $c(w_\e, 1+\frac{2}{N})\leq C$, where the constant $C>0$ depends only on $\al, N$ and $\Omega$. Moreover, $C>0$ is independent of $\e>0$.
\end{proof}
\begin{lemma}\label{08.16.L2}
Let $z\in H_0^1(\Omega;w_\e)$. Then,
\begin{equation}\label{08.20.1}
(N+\al-2)\|w_\e^{\frac{1}{2}-\frac{1}{\al}}z\|_{L^2(\Omega)}\leq 2\|z\|_{H_0^1(\Omega;w_\e)}.
\end{equation}
Moreover,
\begin{equation}\label{08.20.2}
\|z\|_{L^2(\Omega; w_\e)}\leq \frac{2m}{N+\al-2}\||\nabla z|\|_{L^2(\Omega;w_\e)},
\end{equation}
and
\begin{equation}\label{08.31.10}
\|z\|_{L^2(\Omega)}\leq \frac{2m^{1-\frac{\al}{2}}}{N+\al-2}\||\nabla z|\|_{L^2(\Omega;w_\e)}.
\end{equation}
\end{lemma}
\begin{proof}
Since $w_\e\in C^{4}(\overline{\Omega})$, we have
\begin{equation*}
\begin{split}
2\int_{\Omega} w_\e^{1-\frac{2}{\al}}z(x\cdot\nabla z)\df x
&=\int_\Omega w_\e^{1-\frac{2}{\al}}x\cdot \nabla z^2\df x=\int_\Omega \Div\left(w_\e^{1-\frac{2}{\al}}z^2x\right)\df x-\int_\Omega z^2\Div\left(w_\e^{1-\frac{2}{\al}}x\right)\df x\\
&=-(N+\al-2)\int_\Omega z^2w_\e^{1-\frac{2}{\al}} \df x-(2-\al)\int_{B_\e}z^2\psi_\e^{\al-3}[\psi_\e-x\cdot\nabla\psi_\e]\df x.
\end{split}
\end{equation*}
Then,
\begin{equation*}
\begin{split}
(N+\al-2)\int_\Omega z^2w_\e^{1-\frac{2}{\al}}\df x
&\leq -2\int_\Omega w_\e^{1-\frac{2}{\al}}z(x\cdot \nabla z)\df x=2\int_\Omega \left(w_\e^{\frac{1}{2}-\frac{1}{\al}}z\right)\left(w_\e^{\frac{1}{2}-\frac{1}{\al}}x\cdot\nabla z\right)\df x\\
&\leq 2\left(\int_\Omega w_\e^{1-\frac{2}{\al}}z^2\df x\right)^\frac{1}{2}\left(\int_\Omega w_\e^{1-\frac{2}{\al}}|x|^2|\nabla z|^2\df x\right)^\frac{1}{2}\\
&\leq 2\left(\int_\Omega w_\e^{1-\frac{2}{\al}}z^2\df x\right)^\frac{1}{2}\left(\int_\Omega |\nabla z|^2w_\e\df x\right)^\frac{1}{2}
\end{split}
\end{equation*}
by $\al\in (0,2)$, and $\psi_\e-x\cdot\nabla\psi_\e=\frac{3}{8\e^3}(\e^2-|x|^2)^2\geq 0$ on $B_\e$, and Lemma \ref{12.17.L1}. This shows that
\begin{equation*}
(N+\al-2)\|w_\e^{\frac{1}{2}-\frac{1}{\al}}z\|_{L^2(\Omega)}\leq 2\||\nabla z|\|_{L^2(\Omega; w_\e)}.
\end{equation*}
Hence \eqref{08.20.1} holds.
Finally, by \eqref{08.20.1} and  $\e^\al\leq w_\e\leq m^\al$ in $\Omega$, we get \eqref{08.20.2} and \eqref{08.31.10}.
\end{proof}

\textbf{Notations}. For $k\in\mathbb{N}$, we denote $w_k=w_\e$ with $\e=\frac{1}{k}$.
Fix $k\in\mathbb{N}$. Consider the following equation
\begin{equation}\label{08.16.1}
\begin{cases}
\partial_t\widehat{\varphi}_k-\Div(w_k\nabla\widehat{\varphi}_k)=f_k, &\text{in } Q,\\
\widehat{\varphi}_k=0, &\text{on }\partial Q, \\
\widehat{\varphi}_k(0)=\varphi_k, & \text{in }\Omega,
\end{cases}
\end{equation}
where $\varphi_k\in L^2(\Omega)$, $f_k$ is a given function satisfying $f_kw_k^{-1}\in L^2(\Omega; w_k)$.
Denote
\begin{equation*}
W_k=\left\{\varphi\in L^2(0,T; H_{0}^1(\Omega;w_k))\colon \partial_t\varphi\in L^2(0,T; H^{-1}(\Omega;w_k))\right\}.
\end{equation*}
It is clear that $W_k\subset C([0,T]; L^2(\Omega))$ by Lemma \ref{08.15.L1}. Moreover, $W_k\hookrightarrow L^2(Q)$ is compact.
\begin{definition}\label{08.16.D1}
We call $\varphi_k\in W_k$ a solution of \eqref{08.16.1}, if
\begin{equation*}
-\iint_Q\varphi_k\partial_t\psi\df x\df t+\iint_Q\left(\nabla \varphi_k\cdot \nabla \psi\right)w_k\df x\df t=\iint_Q f_k\psi\df x\df t+\int_\Omega \varphi_k(x)\psi(x,0)\df x
\end{equation*}
for any $\psi\in W_k$ with $\psi(T)=0$.
\end{definition}
\begin{lemma}\label{08.22.L1}
Let $\varphi_k\in W_k\ (k\in\mathbb{N})$ be the solution of \eqref{08.16.1} with $\varphi_0\in L^2(\Omega)$ and $f_kw_k^{-1}\in L^2(\Omega; w_k)$. Then,
\begin{equation*}
\max_{t\in [0,T]}\|\varphi_k(t)\|_{L^2(\Omega)}+\|\varphi_k\|_{L^2(0,T; H_0^1(\Omega;w_k)}\leq C\left(\|\varphi_0\|_{L^2(\Omega)}+\|f_kw_k^{-1}\|_{L^2(\Omega; w_k)}\right),
\end{equation*}
where the constant $C>0$ depends only on $\al,  N$ and $\Omega$. Furthermore, if $f_k\in L^2(Q)$, then
\begin{equation*}
\max_{t\in [0,T]}\|\varphi_k(t)\|_{L^2(\Omega)}+\|\varphi_k\|_{L^2(0,T; H_0^1(\Omega;w_k)}\leq C\left(\|\varphi_0\|_{L^2(\Omega)}+\|f_k\|_{L^2(\Omega)}\right).
\end{equation*}
Moreover, if $f_k=0$, then, for all $0\leq t\leq \tau\leq T$,
\begin{equation*}
\|\varphi_k(\tau)\|_{L^2(\Omega)}\leq \|\varphi_k(t)\|_{L^2(\Omega)}.
\end{equation*}
\end{lemma}
\begin{proof}
Let $\varphi_k\in W_k$ be the test function, multiplying $\varphi_k$ on both sides of \eqref{08.16.1}, and integrating on $(0,t)$, we obtain
\begin{equation*}
\begin{split}
&\frac{1}{2}\int_\Omega \varphi_k^2(t)\df x+\iint_{\Omega\times (0,t)} |\nabla\varphi_k|^2w_k\df x\df t\\
&=\iint_{\Omega\times (0,t)} f_k\varphi_k\df x+\frac{1}{2}\int_\Omega \varphi_0^2\df x\\
&\leq \delta \iint_{\Omega\times (0,t)}  \varphi_k^2w_k\df x\df t+\frac{1}{2\delta}\iint_{\Omega\times (0,1)}f_k^2w_k^{-1}\df x\df t+\frac{1}{2}\int_\Omega \varphi_0^2\df x.
\end{split}
\end{equation*}
By \eqref{08.20.2} in Lemma \ref{08.16.L2}, taking $\delta=\frac{N+\al-2}{4m}$, we have
\begin{equation*}
\int_\Omega \varphi_k^2(t)\df x+\iint_{\Omega\times (0,t)}|\nabla\varphi_k|^2w_k\df x\df t\leq \frac{2m}{N+\al-2}\iint_{\Omega\times (0,t)} f_k^2w_k^{-1}\df x\df t+\int_\Omega \varphi_0^2\df x.
\end{equation*}
Taking $C=\max\{\frac{2m}{N+\al-2}, 1\}$, we prove the first part of Lemma \ref{08.22.L1}. The other parts are similar to Lemma \ref{12.09.L1} and Corollary \ref{12.11.C1}.
\end{proof}
\begin{theorem}\label{08.16.T1}
Let $\widehat{\varphi}_0\in W$ be the solution of \eqref{12.14.1} with initial data  $\varphi_0\in L^2(\Omega)$ and $fw^{-1}\in L^2(\Omega;w)$. Let $\widehat{\varphi}_k\ (k\in\mathbb{N})$ be the solution of \eqref{08.16.1} with initial data $\varphi_k=\varphi_0$ and $f_k=f$. Then,
\begin{equation}\label{08.22.1}
\begin{split}
\widehat{\varphi}_k
&\rightharpoonup \widehat{\varphi}_0   \text{ weakly in } L^2(0,T; H_0^1(\Omega;w)), \\
\widehat{\varphi}_k
&\rightharpoonup \widehat{\varphi}_0  \text{ weakly in } L^2(Q).
\end{split}
\end{equation}
Moreover,
\begin{equation}\label{12.09.5}
\widehat{\varphi}_k(T)\rightharpoonup \widehat{\varphi}_0(T)   \text{ weakly in } L^2(\Omega).
\end{equation}
\end{theorem}
\begin{proof} The proof is split into four  steps.

{\it Step 1.}  Note that
\begin{equation}\label{12.11.1}
\begin{split}
\iint_Q (fw_k^{-1})^2w_k\df x\df t=\iint_{Q}(fw^{-1})^2(ww_k^{-1})w\df x\df t\leq \iint_Q (fw^{-1})^2w\df x\df t
\end{split}
\end{equation}
and
\begin{equation*}
\iint_{\Omega\times (0,t)}|\nabla\widehat{\varphi}_k|^2w\df x\df t\leq \iint_{\Omega\times (0,t)}|\nabla \widehat{\varphi}_k|^2w_k\df x\df t
\end{equation*}
since $w\leq w_k$ in $\Omega$. By Lemma \ref{08.22.L1}, we get
\begin{equation*}
\begin{split}
\max_{t\in [0,T]}\|\widehat{\varphi}_k(t)\|_{L^2(\Omega)}+\|\widehat{\varphi}_k\|_{L^2(0,T; H_0^1(\Omega;w))}\leq C\left(\|\varphi_0\|_{L^2(\Omega)}+\|fw^{-1}\|_{L^2(\Omega; w)}\right),
\end{split}
\end{equation*}
where the constant $C>0$ depends only on $\al, N$ and $\Omega$. Then,  there exists a subsequence of $\{\widehat{\varphi}k\}{k\in\mathbb{N}}$, still denoted by itself, and $\widetilde{\varphi}_0\in L^2(0,T; H_0^1(\Omega;w))$, such that
\begin{equation}\label{08.16.4}
\begin{split}
\widehat{\varphi}_k
&\rightharpoonup \widetilde{\varphi}_0 \quad \text{weakly in } L^2(0,T; H_0^1(\Omega;w)),
\end{split}
\end{equation}
and by Lemma \ref{08.15.L1}, we have
\begin{equation}\label{08.23.1}
\widehat{\varphi}_k
\rightharpoonup \widetilde{\varphi}_0   \text{ weakly in } L^2(Q).
\end{equation}
Moreover, we have
\begin{equation*}
\widehat{\varphi}_k
\rightharpoonup \widetilde{\varphi}_0   \text{ weakly in } L^2(Q; w).
\end{equation*}
{\it Step 2}. Now, we prove that $\widetilde{\varphi}_0$ is a weak solution of \eqref{12.14.1}.
Suppose $\psi\in W, \psi(T)=0$. By a density argument, we can suppose $\psi\in C^\infty(\overline{Q})$ and $\psi(\cdot, t)$ is compactly supported in $\Omega$ for any $t$. By the definitions of weak solution (see Definitions \ref{08.16.D1}) and  \eqref{08.16.4} and \eqref{08.23.1}, we only need to show
\begin{equation}\label{08.16.5}
\iint_Q(\nabla\widehat{\varphi}_k\cdot \nabla\psi)w_k\df x\df t\to \iint_Q(\nabla\widetilde{\varphi}_0\cdot\nabla\psi)w\df x\df t \quad \text{as } k\to \infty.
\end{equation}
First, for arbitrary $\ga>0$, by \eqref{08.16.4},  there exists $k_0\in\mathbb{N}$, such that for all $k\geq k_0$, we have
\begin{equation}\label{08.16.6}
\left|\iint_Q(\nabla\widehat{\varphi}_k\cdot\nabla\psi)w\df x\df t-\iint_Q(\nabla\widetilde{\varphi}_0\cdot\nabla\psi)w\df x\df t\right|<\frac{1}{2}\ga.
\end{equation}
Second, note that $w=w_k$ on $\Omega\setminus B_{\frac{1}{k}}$, i.e.,
\begin{equation}\label{08.16.7}
\iint_{(\Omega\setminus B_{\frac{1}{k}})\times (0,T)}(\nabla\widehat{\varphi}_k\cdot\nabla\psi)w_k\df x\df t=\iint_{(\Omega\setminus B_{\frac{1}{k}})\times (0,T)}(\nabla \widehat{\varphi}_k\cdot\nabla\psi)w\df x\df t.
\end{equation}
Third, by the same argument as \eqref{12.11.1} in Step 1, we have
\begin{equation}\label{08.16.9}
\begin{split}
&\left|\iint_{B_{\frac{1}{k}}\times (0,T)}(\nabla\widehat{\varphi}_k\cdot\nabla\psi)w_k\df x\df t\right|\\
&\leq \left(\iint_{B_{\frac{1}{k}}\times (0,T)}|\nabla\widehat{\varphi}_k|^2w_k\df x\df t\right)^\frac{1}{2}\left(\iint_{B_{\frac{1}{k}}\times (0,T)}|\nabla\psi|^2w_k\df x\df t\right)^\frac{1}{2}\\
&\leq C\sqrt{T} \left(\sup_{Q}|\nabla\psi|\right)\left(\|\varphi_0\|_{L^2(\Omega)}+\|fw^{-1}\|_{L^2(Q; w)}\right)\left(w_k(B_{\frac{1}{k}})\right)^\frac{1}{2}\\
&\leq C_\psi\sqrt{T} \left(\|\varphi_0\|_{L^2(\Omega)}+\|fw^{-1}\|_{L^2(\Omega; w)}\right)\frac{1}{k^\frac{N+\al}{2}}
\end{split}
\end{equation}
according to
\begin{equation*}
w_k(B_{\frac{1}{k}})=\int_{B_{\frac{1}{k}}}w_k\df x\leq 2^\al\int_{B_{\frac{1}{k}}}\frac{1}{k^\al} \df x= C\frac{1}{k^{N+\al}}
\end{equation*}
from Lemma \ref{12.17.L1},
where the constant $C_\psi$ is a constant that depends only on $\al, N, \Omega$ and $\psi$.
Fourth,
\begin{equation}\label{08.16.10}
\begin{split}
&\left|\iint_{B_{\frac{1}{k}}\times (0,T)}(\nabla\widehat{\varphi}_k\cdot\nabla\psi)w\df x\df t\right|\\
&\leq \left(\iint_{B_{\frac{1}{k}}\times (0,T)}|\nabla\widehat{\varphi}_k|^2w\df x\df t\right)^\frac{1}{2}\left(\iint_{B_{\frac{1}{k}}\times (0,T)}|\nabla\psi|^2w\df x\df t\right)^\frac{1}{2}\\
&\leq C\sqrt{T} \left(\sup_{Q}|\nabla\psi|\right)\left(\|\varphi_0\|_{L^2(\Omega)}+\|f\|_{L^2(Q; w)}\right)\left(w(B_{\frac{1}{k}})\right)^\frac{1}{2}\\
&\leq C_\psi\sqrt{T} \left(\|\varphi_0\|_{L^2(\Omega)}+\|f\|_{L^2(Q; w)}\right)\frac{1}{k^\frac{N+\al}{2}}
\end{split}
\end{equation}
according to
\begin{equation*}
w(B_{\frac{1}{k}})=\int_{B_{\frac{1}{k}}}w\df x\leq  \int_{B_{\frac{1}{k}}}|x|^\al \df x=C \frac{1}{k^{N+\al}},
\end{equation*}
where the constant $C_\psi$ depends only on $\al, N,\Omega$ and $\psi$.
Finally, since $w=w_k$ on $\Omega\setminus B_{\frac{1}{k}}$, by  \eqref{08.16.6}, \eqref{08.16.7}, \eqref{08.16.9} and \eqref{08.16.10}, we have
\begin{equation*}
\begin{split}
&\left|\iint_Q(\nabla\widehat{\varphi}_k\cdot \nabla\psi)w_k\df x\df t- \iint_Q(\nabla \widetilde{\varphi}_0\cdot\nabla\psi)w\df x\df t\right|\\
&\leq  \left|\iint_Q(\nabla\widehat{\varphi}_k\cdot \nabla\psi)w_k\df x\df t-\iint_Q(\nabla\widehat{\varphi}_k\cdot \nabla\psi)w\df x\df t\right|\\
&\hspace{4.5mm}+\left|\iint_Q(\nabla\widehat{\varphi}_k\cdot\nabla\psi)w\df x\df t-\iint_Q(\nabla\widetilde{\varphi}_0\cdot\nabla\psi)w\df x\df t\right|\\
&\leq \left|\iint_{(\Omega\setminus B_{\frac{1}{k}})\times (0,T)}(\nabla\widehat{\varphi}_k\cdot\nabla\psi)w_k\df x\df t-\iint_{(\Omega\setminus B_{\frac{1}{k}})\times (0,T)}(\nabla\widehat{\varphi}_k\cdot\nabla\psi)w\df x\df t\right|\\
&\hspace{4.5mm}+\left|\iint_{B_{\frac{1}{k}}\times (0,T)}(\nabla\widehat{\varphi}_k\cdot\nabla\psi)w_k\df x\df t\right|+\left|\iint_{B_{\frac{1}{k}}\times (0,T)}(\nabla\widehat{\varphi}_k\cdot\nabla\psi)w\df x\df t\right|+\frac{1}{2}\ga\\
&\leq C_\psi\sqrt{T}\left(\|\varphi_0\|_{L^2(\Omega)}+\|f\|_{L^2(Q; w)}\right)\frac{1}{k^\frac{N+\al}{2}}+\frac{1}{2}\ga,
\end{split}
\end{equation*}
and hence \eqref{08.16.5} holds. From this and Step 1, we have proved
\begin{equation*}
-\iint_Q\widetilde{\varphi}_0\partial_t\psi\df x\df t+\iint_Q (\nabla\widetilde{\varphi}_0\cdot \nabla\psi) w\df x\df t=\iint_Q f\psi\df x\df t+\int_\Omega \varphi_0(x)\psi(x,0)\df x.
\end{equation*}
This shows that $\widetilde{\varphi}_0$ is a solution of \eqref{12.14.1}.

{\it Step 3}. Since $\widetilde{\varphi}_0$ and $\widehat{\varphi}_0$ are solutions of \eqref{12.14.1} with initial data $\varphi_0\in L^2(\Omega)$ and $fw^{-1}\in L^2(\Omega; w)$, by the uniqueness of the solution of \eqref{12.14.1}, we have $\widetilde{\varphi}_0=\widehat{\varphi}_0$.

{\it Step 4}. Finally, let $\psi\in C^\infty(\overline{Q})$ and $\psi(\cdot, t)$ compactly supported in $\Omega$ for any $t\in [0,T]$. Multiplying $\psi$ on both sides of \eqref{08.16.1} with $\widehat{\varphi}_k(0)=\varphi_0$ and $f_k=f$, then
\begin{equation*}
\begin{split}
\int_\Omega \widehat{\varphi}_k(T)\psi(T)\df x-\iint_Q \widehat{\varphi}_k\partial_t\psi\df x\df t+\iint_Q w_k\nabla \widehat{\varphi}_k\cdot \nabla\psi\df x\df t=\iint_Q f\psi\df x\df t+\int_\Omega \widehat{\varphi}_k(0)\psi(0)\df x.
\end{split}
\end{equation*}
Multiplying $\psi$  on both sides of \eqref{12.14.1}, we obtain
\begin{equation*}
\int_\Omega \widehat{\varphi}_0(T)\psi(T)\df x-\iint_Q \widehat{\varphi}_0\partial_t\psi\df x\df t+\iint_Q w\nabla \widehat{\varphi}_0\cdot \nabla\psi\df x\df t=\iint_Q f\psi\df x\df t+\int_\Omega \widehat{\varphi}_0(0)\psi(0)\df x.
\end{equation*}
By \eqref{08.16.5} and $\widehat{\varphi}_0=\widetilde{\varphi}_0$, $\widehat{\varphi}_k(0)=\widehat{\varphi}_0(0)=\varphi_0$ for all $k\in\mathbb{N}$, and $\widehat{\varphi}_k\rightharpoonup \widehat{\varphi}_0$ weakly in $L^2(Q)$, we obtain
\begin{equation*}
\int_\Omega \widehat{\varphi}_k(T)\psi(T)\df x\to \int_\Omega \widehat{\varphi}_0(T)\psi(T)\df x.
\end{equation*}
This shows that $\widehat{\varphi}_k(T)\rightharpoonup \widehat{\varphi}_0(T)$ in $L^2(\Omega)$ since $\psi(T)$ is arbitrary.
\end{proof}
\begin{corollary}\label{08.23.C1}
Under the assumptions in Theorem  \ref{08.16.T1}. If we further assume the initial data $\varphi_0\in H^1(\Omega)$ and $f_k=0$ for all $k\in\mathbb{N}$, then
\begin{equation*}
\widehat{\varphi}_k\to \widehat{\varphi}_0 \quad \text{strongly in } L^2(Q),
\end{equation*}
and
\begin{equation*}
\widehat{\varphi}_k(T)\to \widehat{\varphi}_0(T) \quad \text{strongly in }L^2(\Omega).
\end{equation*}
\end{corollary}
\begin{proof} The proof is split into two steps.

{\it Step 1}. Since $\varphi_0\in H_0^1(\Omega)$, the solution $\widehat{\varphi}_k$ of \eqref{08.16.1} satisfies (note that $H_{w_k,0}^1=H_0^1(\Omega)$ for any $k\in\mathbb{N}$)
\begin{equation}\label{12.25.1}
\widehat{\varphi}k\in L^2(0,T; H^2(\Omega)\cap H{0}^1(\Omega))\cap H^1(0,T; H_0^1(\Omega)) \quad \text{and} \quad \partial_t\widehat{\varphi}_k\in L^2(0,T; L^2(\Omega)).
\end{equation}
Multiplying $\partial_t\widehat{\varphi}_k$ on both sides of \eqref{08.16.1}, integrating on $\Omega\times(0,t)$, we get
\begin{equation*}
\begin{split}
\iint_{\Omega\times(0,t)}(\partial_t\widehat{\varphi}_k)^2\df x\df t+\frac{1}{2}\iint_{\Omega\times(0,t)} \partial_t\left(|\nabla \widehat{\varphi}_k|^2w_k\right)\df x\df t=\iint_{\Omega\times(0,t)} f\partial_t\widehat{\varphi}_k\df x\df t,
\end{split}
\end{equation*}
which implies that
\begin{equation*}
\begin{split}
&\iint_{\Omega\times(0,t)} |\partial_t\widehat{\varphi}_k|^2\df x\df t+\frac{1}{2}\int_\Omega |\nabla\widehat{\varphi}_k(t)|^2w_k\df x\\
&=\iint_Q f\partial_t\widehat{\varphi}_k\df x\df t+\frac{1}{2}\int_\Omega |\nabla \widehat{\varphi}_k(0)|^2w_k\df x\\
&\leq \frac{1}{2}\iint_Q |\partial_t\widehat{\varphi}_k|^2\df x\df t+\frac{m^\al}{2}\left(\iint_Q (fw^{-1})^2w\df x\df t+\int_\Omega |\nabla \varphi_0|^2\df x\right)
\end{split}
\end{equation*}
for all $t\in (0,T]$.
Hence
\begin{equation}\label{12.09.3}
\iint_Q|\partial_t\widehat{\varphi}_k|^2\df x\df t+\int_\Omega |\nabla\widehat{\varphi}_k(t)|^2w_k\df x\leq m^\al \left(\iint_Q (fw^{-1})^2w\df x\df t+\int_\Omega |\nabla \varphi_0|^2\df x\right).
\end{equation}
Therefore, by \eqref{08.15.10}, we have
\begin{equation}\label{08.23.2}
\begin{split}
\|\partial_t\widehat{\varphi}_k\|_{L^2(0,T; H_w^{-1}(\Omega))}
&=\sup_{\|\psi\|_{L^2(0,T; H_0^1(\Omega;w))}\leq 1}\langle \partial_t\widehat{\varphi}_k, \psi\rangle_{L^2(0,T; H_w^{-1}(\Omega)), L^2(0,T; H_0^1(\Omega;w))}\\
&=\sup_{\|\psi\|_{L^2(0,T; H_0^1(\Omega;w))}\leq 1}\iint_Q (\partial_t\widehat{\varphi}_k) \psi\df x\df t\\
&\leq  C\sup_{\|\psi\|_{L^2(0,T; H_0^1(\Omega;w))}\leq 1}\|\partial_t\widehat{\varphi}_k\|_{L^2(Q)} \|\psi\|_{L^2(Q)}\leq C\|\partial_t\widehat{\varphi}_k\|_{L^2(Q)},
\end{split}
\end{equation}
where the constants $C>0$ depends only on $\al, N$ and $\Omega$, and we used \eqref{08.15.10} in the last inequality. From which and Theorem \ref{08.16.T1}, we obtain that there exists a subsequence of $\{\widehat{\varphi}_k\}$, still denoted by itself,  such that
\begin{equation*}
\partial_t\widehat{\varphi}_k\rightharpoonup \partial_t\widehat{\varphi}_0   \text{ weakly in } L^2(0,T; H_w^{-1}(\Omega)).
\end{equation*}
From this,  \eqref{08.23.2},  and the embedding $W\hookrightarrow L^2(Q)$ is compact, we have
\begin{equation*}
\widehat{\varphi}_k\to \widehat{\varphi}_0   \text{ strongly in }L^2(Q).
\end{equation*}
{\it Step 2}. Next, by \eqref{12.09.3} and $w\leq w_k$ for all $k\in\mathbb{N}$, we have
\begin{equation*}
\int_\Omega|\nabla \widehat{\varphi}_k(T)|^2w\df x\leq \int_\Omega|\nabla \widehat{\varphi}_k(T)|^2w_k\df x\leq m^\al \left(\iint_Q (fw^{-1})^2w\df x\df t+\int_\Omega|\nabla \varphi_0|^2\df x\right).
\end{equation*}
This show that there exists a subsequence of $\{\widehat{\varphi}_k(T)\}$, still denoted by itself, and $\xi\in H_0^1(\Omega;w)$, such that $\widehat{\varphi}_k(T)\rightharpoonup \xi$ weakly in $H_0^1(\Omega;w)$, moreover, we have $\widehat{\varphi}_k(T)\to \xi$ strongly in $L^2(\Omega)$ by Lemma \ref{08.15.L4}. This, together with \eqref{12.09.5},  we get $\xi=\widehat{\varphi}_0(T)$.
This completes the proof of the corollary.
\end{proof}

\section{Carleman estimate}\label{S4}

 In this section, we aim to derive the Carleman estimate stated in Theorem \ref{06.10.T2} for the backward approximate equations \eqref{06.08.1}. As noted in Remark \ref{05.16.R1}, it remains unclear how to directly obtain the Carleman estimate from the backward equation \eqref{06.08.2} associated with equation \eqref{12.14.1}.
\begin{lemma}\label{12.26.L1}
Let $\Omega \subset \mathbb{R}^N$ be a $C^3$ bounded domain containing the origin, i.e., $0 \in \Omega$. Let $R > 0$ be such that $B_{4R} \subset \Omega$. Define an open subset $\widehat{\omega} = B_{2R}(x_0) \subset B_{3R}(x_0) \subset \Omega$ satisfying $B_{3R} \cap B_{3R}(x_0) = \emptyset$. Then, there exists a function $\eta \in C^{2,1}(\overline{\Omega})$ such that $\eta > 0$ in $\Omega$, $\eta = 0$ on $\partial \Omega$,
\begin{equation*}
\eta(x) = \psi_\varepsilon^{2-\alpha}(x) \quad \text{for } x \in B_{2R},
\end{equation*}
and
\begin{equation*}
|\nabla \eta| \geq C > 0 \quad \text{on } \overline{\Omega \setminus (B_{R} \cup \widehat{\omega})},
\end{equation*}
where $C > 0$ is an absolute constant.
\end{lemma}
\begin{proof}
Define $\eta = \psi_\varepsilon^{2-\alpha}$ for $x \in B_{2R}$. We construct $\eta \in C^{2,1}(\overline{\Omega})$ such that $\eta > 0$ in $\Omega$ and $\eta = 0$ on $\partial \Omega$, leveraging the conditions $B_{3R} \cap B_{3R}(x_0) = \emptyset$ and $B_{4R} \subset \Omega$. By \cite[Proposition 3.31, Chapter 3.6]{Rousseau}, we can ensure that $|\nabla \eta| \geq C > 0$ on $B_{2R} \setminus B_{\frac{1}{2}R}$, which facilitates the transfer of critical points from $\overline{\Omega \setminus B_R}$ to $\widehat{\omega}$.
\end{proof}
\begin{remark}
We observe that the origin $0 \in \Omega$ is a critical point of $\eta$ in Lemma \ref{12.26.L1}.
\end{remark}
We now proceed to establish the Carleman estimate for the backward equation \eqref{12.14.1}:
\begin{equation}\label{06.08.2}
\begin{cases}
\partial_t u + \Div(w \nabla u) = g, & \text{in } Q, \\
u = 0, & \text{on } \partial Q, \\
u(T) = u_T, & \text{in } \Omega,
\end{cases}
\end{equation}
where $g \in L^2(\Omega)$ and $u_T \in L^2(\Omega)$.
In accordance with Theorem \ref{08.16.T1} and Corollary \ref{08.23.C1}, our initial step involves obtaining the Carleman estimate for the approximate uniformly parabolic equations corresponding to \eqref{06.08.2}. To achieve this, we consider the backward equation \eqref{06.08.1} associated with equation \eqref{08.16.1}.

\subsection{Carleman estimate of the approximate backward uniformly parabolic equations}\label{S4.1}
 Let $0 < \varepsilon \ll R$. Consider the following approximate system of equations:
\begin{equation}\label{06.08.1}
\begin{cases}
\partial_t u + \Div(w_\varepsilon \nabla u) = g, &\text{in } Q, \\
u = 0, &\text{on } \partial Q, \\
u(T) = u_T, &\text{in } \Omega,
\end{cases}
\end{equation}
where $g \in L^2(\Omega)$ and $u_T \in L^2(\Omega)$. We observe that the equations in \eqref{06.08.1} are uniformly parabolic, ensuring that the second-order partial derivatives are well-defined everywhere.
Next, we define the following functions:
\begin{equation*}
\Theta(t) = \frac{1}{[t(T - t)]^4}, \quad \xi(x,t) = \Theta(t) e^{\lambda(8|\eta|_\infty + \eta(x))}, \quad \sigma(x,t) = \Theta(t) e^{10\lambda|\eta|_\infty} - \xi(x,t),
\end{equation*}
where $|\eta|\infty = \esssup{x \in \Omega} |\eta(x)|$. Consequently, we have 
\begin{equation}\label{06.08.3}
\nabla \xi = \lambda \xi \nabla \eta, \quad \nabla \sigma = -\nabla \xi = -\lambda \xi \nabla \eta, \quad \Div(w_\varepsilon \nabla \sigma) = -\lambda^2 \xi w_\varepsilon |\nabla \eta|^2 - \lambda \xi \Div(w_\varepsilon \nabla \eta),
\end{equation}
and
\begin{equation}\label{06.08.4}
\nabla (\partial_t \sigma) = \partial_t \nabla \sigma = -\lambda (\partial_t \xi) \nabla \eta.
\end{equation}
We introduce the transformation:
\begin{equation}\label{06.08.5}
v = e^{-s\sigma} u \quad \Leftrightarrow \quad u = e^{s\sigma} v.
\end{equation}
This transformation implies the following properties:
\begin{equation}\label{gbz2}
\begin{array}{l}
 i)\;\; v = \nabla v = 0  \hbox{ at } t = 0 \hbox{  and } t = T, \\
ii)\;\;   v = 0 \hbox{ on } \partial Q.
\end{array}
\end{equation}

By direct computation, we obtain:
\begin{equation*}
\begin{split}
e^{-s\sigma} g
&= v_t + s w_\varepsilon \nabla v \cdot \nabla \sigma + s \Div(v w_\varepsilon \nabla \sigma) \\
&\quad + \Div(w_\varepsilon \nabla v) + s (\partial_t \sigma) v + s^2 v w_\varepsilon \nabla \sigma \cdot \nabla \sigma,
\end{split}
\end{equation*}
which can be rewritten as:
\begin{equation*}
e^{-s\sigma} g = P_1 v + P_2 v,
\end{equation*}
where
\begin{equation}\label{gbz3}
\left\{\begin{array}{l}
\disp P_1 v = v_t + s w_\varepsilon \nabla v \cdot \nabla \sigma + s \Div(v w_\varepsilon \nabla \sigma), \crr\disp
P_2 v = \Div(w_\varepsilon \nabla v) + s (\partial_t \sigma) v + s^2 v w_\varepsilon \nabla \sigma \cdot \nabla \sigma.
\end{array}\right.
\end{equation}

\subsubsection{Computations}\label{S4.1.1}

 Let $P_i$, for $i = 1, 2$, be defined as in \dref{gbz3}. We can express
\begin{equation}\label{gbz4}
(P_1v,P_2v)_{L^2(Q)} = \sum_{i = 1}^4 I_i(Q),
\end{equation}
where
\begin{equation}\label{gbz5}
\left\{\begin{aligned}
I_1(Q) &= \left(v_t, \Div(w_\e\nabla v)+s\si_tv + s^2vw_\e\nabla \si\cdot \nabla \si\right)_{L^2(Q)},\\
I_2(Q) &= s^2\left(\Div(vw_\e \nabla \si)+w_\e\nabla v\cdot \nabla \si,\si_t v\right)_{L^2(Q)},\\
I_3(Q) &= s^3\left( \Div(vw_\e\nabla\si)+w_\e\nabla v\cdot \nabla \si, vw_\e\nabla\si\cdot \nabla\si\right)_{L^2(Q)}, \\
I_4(Q) &= s\left(\Div(v w_\e\nabla\si)+w_\e \nabla v\cdot \nabla \si, \Div(w_\e \nabla v)\right)_{L^2(Q)}.
\end{aligned}\right.
\end{equation}
Next, we compute $I_i(Q)$ for $i = 1, 2, 3, 4$.

\textbf{Computation of $I_1(Q)$.} By \dref{gbz2} and the fact that $v_t = 0$ on $\partial Q$, we have
\begin{align*}
I_1(Q) &= \iint_Q v_t\left(\Div(\psi_\e^\al \nabla v)+s\si_t v + s^2\la^2\xi^2 v\psi_\e^\al |\nabla\eta|^2\right) dx dt\\
&= -\frac{s}{2}\iint_Q v^2\si_{tt} dx dt - s^2\la^2\iint_Q \xi\xi_t v^2\psi_\e^\al |\nabla\eta|^2 dx dt.
\end{align*}

\textbf{Computation of $I_2(Q)$.} We obtain
\begin{align*}
I_2(Q) &= -s^2\iint_Q v\psi_\e^\al \nabla \si\cdot \nabla(\si_t v) dx dt + s^2\iint_Q \si_t v  \psi_\e^\al \nabla\si\cdot \nabla v dx dt\\
&= -s^2\la^2\iint_Q\xi \xi_tv^2\psi_\e^\al |\nabla\eta|^2 dx dt.
\end{align*}

\textbf{Computation of $I_3(Q)$.} We have
\begin{align*}
I_3(Q) &= s^3\iint_Q \left(\Div(v\psi_\e^\al \nabla\si)+\psi_\e^\al\nabla v\cdot \nabla\si\right)v\psi_\e^\al \nabla \si\cdot \nabla\si dx dt\\
&= -s^3\iint_Q v^2\psi_\e^\al \nabla\si\cdot \nabla (\psi_\e^\al \nabla\si\cdot \nabla\si) dx dt\\
&= 2s^3\la^4\iint_Q \xi^3v^2(\psi_\e^{\al}|\nabla\eta|^2)^2 dx dt + s^3\la^3\iint_Q \xi^3v^2\psi_\e^\al \nabla\eta\cdot \nabla (\psi_\e^\al|\nabla\eta|^2) dx dt.
\end{align*}
\textbf{Computation of $I_4(Q)$.} By \dref{gbz2}, we get
\begin{align*}
I_4(Q) &= s\iint_Q \left(\psi_\e^\al \nabla v\cdot \nabla \si+\Div(v\psi_\e^\al \nabla\si)\right) \Div(\psi_\e^\al \nabla v) dx dt\\
&= -2s\la\iint_Q \xi \psi_\e^\al (\nabla v\cdot \nabla\eta) \Div(\psi_\e^\al \nabla v) dx dt - s\la\iint_Q v\Div(\xi\psi_\e^\al \nabla\eta)\Div(\psi_\e^\al \nabla v) dx dt\\
&= s\la\iint_{\partial Q}\xi  \psi_\e^{2\al} \left|\frac{\partial v}{\partial \nu}\right|^2\left|\frac{\partial \eta}{\partial \nu}\right| dS dt + 2s\la^2\iint_Q \xi (\psi_\e^\al \nabla v\cdot \nabla\eta)^2 dx dt\\
&\quad + 2\al s\la\iint_Q \xi \psi_\e^{2\al - 1}(\nabla v\cdot \nabla\eta)(\nabla v\cdot \nabla\psi_\e) dx dt + 2s\la\iint_Q \xi \psi_\e^{2\al}(D^2\eta\nabla v)\cdot \nabla v dx dt\\
&\quad - 2\al s\la\iint_Q \xi \psi_\e^{2\al - 1}|\nabla v|^2(\nabla\eta\cdot \nabla\psi_\e) dx dt - s\la\iint_Q\xi \psi_\e^{2\al}|\nabla v|^2\De\eta dx dt\\
&\quad + s\la\iint_Q \xi \psi_\e^\al |\nabla v|^2\Div(\psi_\e^\al \nabla\eta) dx dt\\
&\quad + s\la^3\iint_Q \xi\psi_\e^{2\al} v(\nabla v\cdot \nabla\eta)|\nabla\eta|^2 dx dt + s\la^2\iint_Q\xi \psi_\e^\al v\nabla v\cdot \nabla [\psi_\e^\al |\nabla\eta|^2] dx dt\\
&\quad + s\la^2\iint_Q \xi \psi_\e^\al v(\nabla v\cdot \nabla \eta)\Div(\psi_\e^\al \nabla\eta) dx dt + s\la\iint_Q\xi v\psi_\e^\al \nabla v\cdot \nabla \Div(\psi_\e^\al \nabla\eta) dx dt,
\end{align*}
where we used the following identities:
\begin{align*}
&-2s\la\iint_Q\xi \psi_\e^\al (\nabla v\cdot \nabla\eta)\Div(\psi_\e^\al \nabla v) dx dt\\
&= 2s\la\iint_{\partial Q}\xi  \psi_\e^{2\al} \left|\frac{\partial v}{\partial \nu}\right|^2\left|\frac{\partial \eta}{\partial \nu}\right| dS dt + 2s\la^2\iint_Q \xi (\psi_\e^\al \nabla v\cdot \nabla\eta)^2 dx dt\\
&\quad + 2\al s\la\iint_Q \xi \psi_\e^{2\al - 1}(\nabla v\cdot \nabla\eta)(\nabla v\cdot \nabla\psi_\e) dx dt + 2s\la\iint_Q \xi \psi_\e^{2\al}\nabla v\cdot \nabla ( \nabla v\cdot \nabla\eta) dx dt\\
&= s\la\iint_{\partial Q}\xi  \psi_\e^{2\al} \left|\frac{\partial v}{\partial \nu}\right|^2\left|\frac{\partial \eta}{\partial \nu}\right| dS dt + 2s\la^2\iint_Q \xi (\psi_\e^\al \nabla v\cdot \nabla\eta)^2 dx dt\\
&\quad + 2\al s\la\iint_Q \xi \psi_\e^{2\al - 1}(\nabla v\cdot \nabla\eta)(\nabla v\cdot \nabla\psi_\e) dx dt + 2s\la\iint_Q \xi \psi_\e^{2\al}(D^2\eta\nabla v)\cdot \nabla v dx dt\\
&\quad - s\la^2\iint_Q\xi \psi_\e^{2\al}|\nabla v|^2|\nabla\eta|^2 dx dt - 2\al s\la\iint_Q \xi \psi_\e^{2\al - 1}|\nabla v|^2(\nabla\eta\cdot \nabla\psi_\e) dx dt\\
&\quad - s\la\iint_Q\xi \psi_\e^{2\al}|\nabla v|^2\De\eta dx dt,
\end{align*}
since $\nabla v\cdot \nabla (\nabla v\cdot \nabla\eta)=(D^2\eta\nabla v)\cdot \nabla v+\frac{1}{2}\nabla \eta\cdot \nabla|\nabla v|^2$, and
\begin{align*}
& s\la\iint_Q \xi \psi_\e^{2\al}\nabla\eta\cdot \nabla|\nabla v|^2 dx dt\\
&= -s\la\iint_{\partial Q}\xi \psi_\e^{2\al}\left|\frac{\partial v}{\partial \nu}\right|^2\left|\frac{\partial\eta}{\partial \nu}\right|dS dt - s\la^2\iint_Q \xi \psi_\e^{2\al}|\nabla v|^2|\nabla\eta|^2 dx dt\\
&\quad - 2\al s\la\iint_Q \xi \psi_\e^{2\al - 1}|\nabla v|^2 (\nabla \eta\cdot \nabla\psi_\e) dx dt - s\la\iint_Q\xi \psi_\e^{2\al}|\nabla v|^2 \De\eta dx dt,
\end{align*}
and
\begin{align*}
&-s\la\iint_Q v\Div(\xi\psi_\e^\al \nabla\eta)\Div(\psi_\e^\al \nabla v) dx dt\\
&= -s\la^2\iint_Q \xi  v\psi_\e^\al |\nabla\eta|^2\Div(\psi_\e^\al \nabla v) dx dt - s\la\iint_Q \xi v\Div(\psi_\e^\al \nabla\eta)\Div(\psi_\e^\al \nabla v) dx dt\\
&= s\la^2\iint_Q \psi_\e^\al \nabla v\cdot\nabla \left[\xi v\psi_\e^\al |\nabla\eta|^2\right] dx dt + s\la\iint_Q \psi_\e^\al \nabla v\cdot \nabla \left[\xi v\Div(\psi_\e^\al \nabla\eta)\right] dx dt\\
&= s\la^2\iint_Q \xi  \psi_\e^{2\al} |\nabla v|^2|\nabla\eta|^2 dx dt + s\la\iint_Q \xi \psi_\e^\al |\nabla v|^2\Div(\psi_\e^\al \nabla\eta) dx dt\\
&\quad + s\la^3\iint_Q \xi\psi_\e^{2\al} v(\nabla v\cdot \nabla\eta)|\nabla\eta|^2 dx dt + s\la^2\iint_Q\xi \psi_\e^\al v\nabla v\cdot \nabla [\psi_\e^\al |\nabla\eta|^2] dx dt\\
&\quad + s\la^2\iint_Q \xi \psi_\e^\al v(\nabla v\cdot \nabla \eta)\Div(\psi_\e^\al \nabla\eta) dx dt + s\la\iint_Q\xi v\psi_\e^\al \nabla v\cdot \nabla \Div(\psi_\e^\al \nabla\eta) dx dt,
\end{align*}
which follows from condition ii) of \dref{gbz2}.

\subsubsection{Estimations}\label{S4.1.2}

Let $P_i$, for $i = 1, 2$, be defined as in \dref{gbz3}.  From equations (1)-(4) in Section \ref{S4.1.1}, we know that
$$(P_1v,P_2v)_{L^2(Q)}=\sum_{i=1}^4I_i(Q)$$
contains a boundary term given by
\begin{equation*}
s\la\iint_{\pt Q}\xi  \psi_\e^{2\al} \left|\f{\pt v}{\pt \nu}\right|^2\left|\f{\pt \eta}{\pt \nu}\right| \df S\df t\geq 0,
\end{equation*}
along with other integral terms over $Q$. Since the boundary term is positive, we omit it in the subsequent analysis.
We will estimate $I_1(Q),I_2(Q),I_3(Q),I_4(Q)$ by partitioning $Q$ into three parts: $Q=Q_{2R}^c\cup Q_{\e,2R}\cup Q_\e$, where
\begin{equation*}
Q_{2R}^c=(\Om\se B_{2R})\ts (0,T),\quad Q_{\e,2R}=(B_{2R}\se B_\e)\ts (0,T),\quad Q_\e=B_\e\ts (0,T).
\end{equation*}
Note that
  \begin{equation*}
	\begin{split}
		|\si_{tt}|\leq C\xi^\f{3}{2}, \quad |\xi\xi_t|\leq C\xi^\f{9}{4}\leq C\xi^3,
	\end{split}
\end{equation*}
where the constant $C>0$ depends only on $T$. Furthermore, the following estimates will rely on \eqref{12.18.1} and \eqref{12.18.2}.

\textbf{1)~ Estimation on $Q_{2R}^c$.} In this region, $\eta\in C^{2,1}(\ol\Om)$, $\eta=0$ on $\pt\Om$, $\eta>0$ on $\Om$, and $|\nabla \eta|\geq C>0$ on $\ol{\Om\se (B_R\cup \wh \om)}$.
We have
\begin{equation*}
\begin{split}
I_1(Q_{2R}^c)\geq -Cs\iint_{Q_{2R}^c} \xi^3v^2\df x\df t-Cs^2\la^2\iint_{Q_{2R}^c} \xi^3 v^2\df x\df t.
\end{split}
\end{equation*}
Similarly,
\begin{equation*}
\begin{split}
I_2(Q_{2R}^c)
&\geq -Cs^2\la^2\iint_{Q_{2R}^c} \xi^3v^2\df x\df t,
\end{split}
\end{equation*}
and
\begin{equation*}
\begin{split}
I_3(Q_{2R}^c)
&\geq +2s^3\la^4\iint_{Q_{2R}^c} \xi^3v^2|\nabla\eta|^4\df x\df t.
\end{split}
\end{equation*}
For $I_4(Q_{2R}^c)$, the third to seventh terms can be absorbed by
\begin{equation*}
\begin{split}
Cs\la\iint_{Q_{2R}^c} \xi |\nabla v|^2\df x\df t,
\end{split}
\end{equation*}
and
\begin{eqnarray*}
&&s\la^3\iint_Q \xi \psi_\e^{2\al}v(\nabla v\cdot \nabla\eta)|\nabla\eta|^2\df x\df t\\
&&\geq -s\la^2\iint_{Q_{2R}^c} \xi (\psi_\e^\al \nabla v\cdot \nabla\eta)^2\df x\df t-Cs\la^4\iint_{Q_{2R}^c} \xi v^2\df x\df t.
\end{eqnarray*}
Additionally, from $\eta\in C^{2,1}(\ol\Om)$, we have
\begin{eqnarray*}
&&s\la^2\iint_{Q_{2R}^c}\xi \psi_\e^\al v\nabla v\cdot \nabla [\psi_\e^\al |\nabla \eta|^2]\df x\df t\\
&&+s\la^2\iint_{Q_{2R}^c} \xi \psi_\e^\al v(\nabla v\cdot \nabla \eta)\Div(\psi_\e^\al \nabla\eta)\df x\df t+s\la\iint_{Q_{2R}^c}\xi v\psi_\e^\al \nabla v\cdot \nabla \Div(\psi_\e^\al \nabla\eta)\df x\df t\\
&&\geq -C\la\iint_{Q_{2R}^c} \xi |\nabla v|^2\df x\df t-Cs^2\la^3\iint_{Q_{2R}^c} \xi v^2\df x\df t.
\end{eqnarray*}
Thus, for $\la\geq 1$ and $s\geq 1$ sufficiently large,
\begin{eqnarray*}
I_4(Q_{2R}^c)
&&\geq s\la\iint_{\pt Q}\xi \psi_\e^{2\al}\left|\f{\pt v}{\pt \nu}\right|^2\left|\f{\pt \eta}{\pt \nu}\right|\df S\df t+s\la^2\iint_{Q_{2R}^c} \xi (\psi_\e^\al \nabla v\cdot \nabla\eta)^2\df x\df t\\
&&\hspace{4.5mm}-C\la\iint_{Q_{2R}^c}\xi |\nabla v|^2\df x\df t-Cs^2\la^4\iint_{Q_{2R}^c}\xi^3 v^2\df x\df t,
\end{eqnarray*}
where the constants $C>0$ depend only on $\al, R, T, N$, and $\Om$.
Overall, we obtain
\begin{equation*}
\begin{split}
(P_1v,P_2v)_{L^2(Q_{2R}^c)}
&\geq s\la^2\iint_{Q_{2R}^c}\xi \left(\psi_\e^\al \nabla v\cdot\nabla\eta\right)^2\df x\df t+2s^3\la^4\iint_{Q_{2R}^c}\xi^3v^2\df x\df t\\
&\hspace{4.5mm}-C\la\iint_{Q_{2R}^c}\xi |\nabla v|^2\df x\df t,
\end{split}
\end{equation*}
where the constants $C>0$ depend only on $\al, R, T, N$, and $\Om$.

\textbf{2)~Estimation on $Q_{\e,2R}$.} In this region, $\psi_\e=|x|$, and $\eta=|x|^{2-\al}$ from Lemma \ref{12.26.L1}. Note that
\begin{equation*}
\nabla \eta=(2-\al)|x|^{-\al}x,
\end{equation*}
and
\begin{equation*}
\f{\pt^2\eta}{\pt x_i\pt x_j}=(2-\al)|x|^{-\al}\left[\de_{ij}-\al |x|^{-2}x_ix_j\right],\quad \De\eta=(2-\al)(N-\al)|x|^{-\al}.
\end{equation*}
We have
\begin{eqnarray*}
I_1(Q_{\e,2R})
&&=-\f{s}{2}\iint_{Q_{\e,2R}} v^2\si_{tt}\df x\df t-(2-\al)^2s^2\la^2\iint_{Q_{\e,2R}} \xi\xi_t v^2|x|^{2-\al}\df x\df t\\
&&\geq -Cs\iint_{Q_{\e,2R}}\xi^\f{3}{2}v^2\df x\df t-Cs^2\la^2\iint_{Q_{\e,2R}} \xi^3v^2|x|^{2-\al}\df x\df t\\
&&\geq -C\iint_{Q_{\e,2R}} \xi|x|^\al |\nabla v|^2\df x\df t-Cs^2\la^2\iint_{Q_{\e,2R}} \xi^3 v^2|x|^{2-\al}\df x\df t\\
&&\hspace{4.5mm}-C\iint_{Q_{2R}^c\cup Q_\e} \xi\psi_\e^\al |\nabla v|^2|\nabla \psi_\e|^2\df x\df t-C\la^2\iint_{Q_{2R}^c\cup Q_\e}\xi^3 v^2\psi_\e^{2-\al}|\nabla\psi_\e|^4\df x\df t\\
&&\hspace{4.5mm}-C\iint_{Q_{2R}^c}\xi  v^2\df x\df t-C\e^{\al-2}\iint_{Q_\e}\xi v^2\df x\df t
\end{eqnarray*}
derived from
\begin{eqnarray*}
&&s\iint_{Q_{\e,2R}}\xi^\f{3}{2}v^2\df x\df t\\
&&=\iint_{Q_{\e,2R}} \left(\xi^\f{1}{2}v|x|^{\f{\al}{2}-1}\right)\left(s\xi v|x|^{1-\f{\al}{2}}\right)\df x\df t\\
&&\leq \iint_{Q_{\e,2R}} \xi v^2|x|^{\al-2}\df x\df t+Cs^2\iint_{Q_{\e,2R}} \xi^3v^2|x|^{2-\al}\df x\df t\\
&&\leq C\iint_{Q_{\e,2R}} \xi\psi_\e^\al |\nabla v|^2\df x\df t+C(\la^2+s^2)\iint_{Q_{\e,2R}} \xi^3 v^2|x|^{2-\al}\df x\df t\\
&&\hspace{4.5mm}+C\iint_{Q_{2R}^c\cup Q_\e} \xi\psi_\e^\al |\nabla v|^2|\nabla \psi_\e|^2\df x\df t+C\la^2\iint_{Q_{2R}^c\cup Q_\e}\xi^3 v^2\psi_\e^{2-\al}|\nabla\psi_\e|^4\df x\df t\\
&&\hspace{4.5mm}+C\iint_{Q_{2R}^c}\xi  v^2\df x\df t+C\e^{\al-2}\iint_{Q_\e}\xi v^2\df x\df t
\end{eqnarray*}
according to
\begin{equation}\label{01.16.1}
\begin{split}
&\iint_{Q_{\e,2R}} \xi v^2|x|^{\al-2}\df x\df t=+\iint_{Q_{\e,2R}}\xi v^2\psi_\e^{\al-2} |\nabla \psi_\e|^4\df x\df t\\
&\leq \iint_Q\xi v^2\psi_\e^{\al-2}|\nabla\psi_\e|^4\df x\df t=+\iint_Q (\xi^\f{1}{2}v|\nabla\psi_\e|^2)^2\psi_\e^{\al-2}\df x\df t\\
&\leq C\iint_Q \psi_\e^\al |\nabla (\xi^\f{1}{2}v|\nabla\psi_\e|^2)|^2\df x\df t\\
&\leq C\iint_{Q_{\e,2R}} \xi\psi_\e^\al |\nabla v|^2\df x\df t+C\la^2\iint_{Q_{\e,2R}} \xi^3 v^2|x|^{2-\al}\df x\df t\\
&\hspace{4.5mm}+C\iint_{Q_{2R}^c\cup Q_\e} \xi\psi_\e^\al |\nabla v|^2|\nabla \psi_\e|^4\df x\df t+C\la^2\iint_{Q_{2R}^c\cup Q_\e}\xi^3 v^2\psi_\e^{2-\al}|\nabla\psi_\e|^6\df x\df t\\
&\hspace{4.5mm}+C\iint_{Q_{2R}^c\cup Q_\e}\xi  v^2\psi_\e^\al (D^2\psi_\e\nabla\psi_\e)\cdot (D^2\psi_\e \nabla \psi_\e)\df x\df t\\
&\leq C\iint_{Q_{\e,2R}} \xi\psi_\e^\al |\nabla v|^2\df x\df t+C\la^2\iint_{Q_{\e,2R}} \xi^3 v^2|x|^{2-\al}\df x\df t\\
&\hspace{4.5mm}+C\iint_{Q_{2R}^c\cup Q_\e} \xi\psi_\e^\al |\nabla v|^2|\nabla \psi_\e|^2\df x\df t+C\la^2\iint_{Q_{2R}^c\cup Q_\e}\xi^3 v^2\psi_\e^{2-\al}|\nabla\psi_\e|^4\df x\df t\\
&\hspace{4.5mm}+C\iint_{Q_{2R}^c}\xi  v^2\df x\df t+C\e^{\al-2}\iint_{Q_\e}\xi v^2\df x\df t
\end{split}
\end{equation}
by Lemma \ref{08.16.L2}, \eqref{12.18.1}, \eqref{12.18.2}, and condition ii) of \dref{gbz2}.
Similarly,
\begin{equation*}
\begin{split}
I_2(Q_{\e,2R})
&=-(2-\al)^2s^2\la^2\iint_{Q_{\e,2R}}\xi\xi_t v^2|x|^{2-\al}\df x\df t\geq -Cs^2\la^2\iint_{Q_{\e,2R}} \xi^3v^2|x|^{2-\al}\df x\df t,
\end{split}
\end{equation*}
and
\begin{equation*}
\begin{split}
I_3(Q_{\e,2R})
&=2(2-\al)^4s^3\la^4\iint_{Q_{\e,2R}}\xi^3v^2|x|^{4-2\al}\df x\df t+(2-\al)^4s^3\la^3\iint_{Q_{\e,2R}} \xi^3v^2|x|^{2-\al}\df x\df t.
\end{split}
\end{equation*}
For $I_4(Q_{\e,2R})$, there exists $s_0\geq 1$ (depending only on $\al$) such that, for all $s\geq s_0$ and $\la\geq 1$, we have
\begin{eqnarray*}
&&I_4(Q_{\e,2R})\\
&&=2(2-\al)^2s\la^2\iint_{Q_{\e,2R}}\xi (\nabla v\cdot x)^2\df x\df t+(2-\al)^2s\la\iint_{Q_{\e,2R}}\xi |x|^\al |\nabla v|^2\df x\df t\\
&&\hspace{4.5mm}+(2-\al)^3s\la^3\iint_{Q_{\e,2R}}\xi v(\nabla v\cdot x)|x|^{2-\al}\df x\df t+(2-\al)^2(N-\al+2)s\la^2\iint_{Q_{\e,2R}} \xi v(\nabla v\cdot x)\df x\df t\\
&&\geq (2-\al)^2s\la^2\iint_{Q_{\e,2R}}\xi (\nabla v\cdot x)^2\df x\df t+(2-\al)^2s\la\iint_{Q_{\e,2R}}\xi |x|^\al |\nabla v|^2\df x\df t\\
&&\hspace{4.5mm}-Cs^2\la^3\iint_{Q_{\e,2R}} \xi^3v^2|x|^{2-\al}\df x\df t-Cs^2\la^4\iint_{Q_{\e,2R}}\xi^3v^2|x|^{4-2\al}\df x\df t\\
&&\hspace{4.5mm}-C\la \iint_{Q_{2R}^c\cup Q_\e} \xi\psi_\e^\al |\nabla v|^2|\nabla \psi_\e|^2\df x\df t-C\la^3\iint_{Q_{2R}^c\cup Q_\e}\xi^3 v^2\psi_\e^{2-\al}|\nabla\psi_\e|^4\df x\df t\\
&&\hspace{4.5mm}-C\la \iint_{Q_{2R}^c}\xi  v^2\df x\df t-C\la \e^{\al-2}\iint_{Q_\e}\xi v^2\df x\df t
\end{eqnarray*}
derived from
\begin{eqnarray*}
&&(2-\al)^3s\la^3\iint_{Q_{\e,2R}} \xi v^2(\nabla v\cdot x)|x|^{2-\al}\df x\df t\\
&&\geq -\f{1}{2}(2-\al)^3\la^2 \iint_{Q_{\e,2R}} \xi (\nabla v\cdot x)^2\df x\df t-Cs^2\la^4\iint_{Q_{\e,2R}} \xi^3v^2|x|^{4-2\al}\df x\df t
\end{eqnarray*}
and
\begin{eqnarray*}
&&(2-\al)^2(N-\al+2)s\la^2\iint_{Q_{\e,2R}}\xi v(\nabla v\cdot x)\df x\df t\\
&&\geq -\f{1}{2}(2-\al)^2s\la^2\iint_{Q_{\e,2R}} \xi(\nabla v\cdot x)^2\df x\df t-Cs\la^2\iint_{Q_{\e,2R}} \xi v^2\df x\df t\\
&&\geq -\f{1}{2}(2-\al)^2s\la^2\iint_{Q_{\e,2R}} \xi(\nabla v\cdot x)^2\df x\df t\\
&&\hspace{4.5mm}-C\la \iint_{Q_{\e,2R}} \xi\psi_\e^\al |\nabla v|^2\df x\df t-Cs^2\la^3\iint_{Q_{\e,2R}} \xi^3 v^2|x|^{2-\al}\df x\df t\\
&&\hspace{4.5mm}-C\la \iint_{Q_{2R}^c\cup Q_\e} \xi\psi_\e^\al |\nabla v|^2|\nabla \psi_\e|^2\df x\df t-C\la^3\iint_{Q_{2R}^c\cup Q_\e}\xi^3 v^2\psi_\e^{2-\al}|\nabla\psi_\e|^4\df x\df t\\
&&\hspace{4.5mm}-C\la \iint_{Q_{2R}^c}\xi  v^2\df x\df t-C\la \e^{\al-2}\iint_{Q_\e}\xi v^2\df x\df t
\end{eqnarray*}
according to
\begin{eqnarray*}
&&Cs\la^2\iint_{Q_{\e,2R}}\xi v^2\df x\df t=+Cs\la^2\iint_{Q_{\e,2R}}\xi v^2|\nabla\psi_\e|^2\df x\df t\\
&&=C \iint_{Q_{\e,2R}}\left(s \la^\f{3}{2}\xi^\f{1}{2} v\psi_\e^{1-\f{\al}{2}}\right)\left(\la^\f{1}{2}\xi^\f{1}{2}v \psi_\e^{\f{\al}{2}-1}|\nabla\psi_\e|^2\right)\df x\df t\\
&&\leq Cs^2\la^3\iint_{Q_{\e,2R}} \xi^3 v^2\psi_\e^{2-\al}\df x\df t+C\la\iint_{Q_{\e,2R}} \xi v^2\psi_\e^{\al-2}|\nabla\psi_\e|^4\df x\df t\\
&&\leq C\la \iint_{Q_{\e,2R}} \xi\psi_\e^\al |\nabla v|^2\df x\df t+Cs^2\la^3\iint_{Q_{\e,2R}} \xi^3 v^2|x|^{2-\al}\df x\df t\\
&&\hspace{4.5mm}+C\la \iint_{Q_{2R}^c\cup Q_\e} \xi\psi_\e^\al |\nabla v|^2|\nabla \psi_\e|^2\df x\df t+C\la^3\iint_{Q_{2R}^c\cup Q_\e}\xi^3 v^2\psi_\e^{2-\al}|\nabla\psi_\e|^4\df x\df t\\
&&\hspace{4.5mm}+C\la \iint_{Q_{2R}^c}\xi  v^2\df x\df t+C\la \e^{\al-2}\iint_{Q_\e}\xi v^2\df x\df t
\end{eqnarray*}
by \eqref{01.16.1}.
Overall, we have
\begin{eqnarray*}
&&(P_1v,P_2v)_{L^2(Q_{\e,2R})}\\
&&\geq 2(2-\al)^4s^3\la^4\iint_{Q_{\e,2R}}\xi^3v^2|x|^{4-2\al}\df x\df t+(2-\al)^4s^3\la^3\iint_{Q_{\e,2R}} \xi^3v^2|x|^{2-\al}\df x\df t\\
&&\hspace{4.5mm}+(2-\al)^2s\la^2\iint_{Q_{\e,2R}}\xi (\nabla v\cdot x)^2\df x\df t+(2-\al)^2s\la\iint_{Q_{\e,2R}}\xi |x|^\al |\nabla v|^2\df x\df t\\
&&\hspace{4.5mm}-C\la \iint_{Q_{2R}^c\cup Q_\e} \xi\psi_\e^\al |\nabla v|^2|\nabla \psi_\e|^2\df x\df t-C\la^3\iint_{Q_{2R}^c\cup Q_\e}\xi^3 v^2\psi_\e^{2-\al}|\nabla\psi_\e|^4\df x\df t\\
&&\hspace{4.5mm}-C\la \iint_{Q_{2R}^c}\xi  v^2\df x\df t-C\la \e^{\al-2}\iint_{Q_\e}\xi v^2\df x\df t,
\end{eqnarray*}
where the constants $C>0$ depend only on $\al, R, T, N$, and $\Om$.

\textbf{3)~Estimation on $Q_\e$.}  In this region, $\psi_\e=\f{3\e}{8}+\f{3}{4\e}|x|^2-\f{1}{8\e^3}|x|^4$ for $|x|<\e$, and $\eta=\psi_\e^{2-\al}$ from Lemma \ref{12.26.L1}. Note that
\begin{equation*}
\nabla \eta=(2-\al)\psi_\e^{1-\al}\nabla\psi_\e, \quad
\end{equation*}
and
\begin{equation*}
\f{\pt^2\eta}{\pt x_i\pt x_j}=(2-\al)\psi_\e^{1-\al}\f{\pt^2\psi_\e}{\pt x_i\pt x_j}+(2-\al)(1-\al)\psi_\e^{-\al}\f{\pt\psi_\e}{\pt x_i}\f{\pt \psi_\e}{\pt x_j},
\end{equation*}
and
\begin{equation*}
\De\eta =(2-\al)\psi_\e^{1-\al}\De\psi_\e+(2-\al)(1-\al)\psi_\e^{-\al} |\nabla\psi_\e|^2.
\end{equation*}
Moreover, we have
\begin{equation}\label{01.17.1}
\psi_\e\left(\f{3}{2\e}-\f{1}{2\e^3}|x|^2\right)- |\nabla\psi_\e|^2=\f{3}{16\e^6}(\e^2-|x|^2)^2(3\e^2-|x|^2)\geq 0
\end{equation}
and
\begin{equation}\label{01.17.2}
(\nabla v\cdot \nabla \psi_\e)^2-\f{1}{\e^3}\psi_\e(\nabla v\cdot x)^2=(\nabla v\cdot x)^2\left[\f{3}{8\e^6}(\e^2-|x|^2)(5\e^2-|x|^2)\right]\geq 0,
\end{equation}
and
\begin{equation}\label{01.17.3}
\begin{split}
& \nabla\left[|\nabla \psi_\e|^2+\psi_\e\De\psi_\e \right]=\f{3}{8\e^6}\left[5(2+N)\e^4-6(4+N)\e^2|x|^2+(6+N)|x|^4\right]x,\\
& -\f{3}{\e^2}\leq   \f{3}{8\e^6}\left[5(2+N)\e^4-6(4+N)\e^2|x|^2+(6+N)|x|^4\right]\leq \f{15(N+2)}{8\e^2}.
\end{split}
\end{equation}
We have
\begin{equation*}
\begin{split}
I_1(Q_\e)
&\geq -Cs\iint_{Q_\e}\xi^\f{3}{2}v^2\df x\df t-Cs^2\la^2\iint_{Q_\e}\xi^\f{9}{4} v^2\psi_\e^{2-\al}|\nabla \psi_\e|^2\df x\df  t\\
&\geq -Cs\iint_{Q_\e}\xi^\f{3}{2}v^2\df x\df t-Cs^2\la^2\e^{2-\al}\iint_{Q_\e}\xi^\f{9}{4} v^2\df x\df  t.
\end{split}
\end{equation*}
Similarly,
\begin{equation*}
\begin{split}
I_2(Q_\e)\geq -Cs^2\la^2 \iint_{Q_\e}\xi^\f{9}{4} v^2\psi_\e^{2-\al}|\nabla\psi_\e|^2\df x\df t\geq -Cs^2\la^2\e^{2-\al}\iint_{Q_\e}\xi^\f{9}{4}v^2\df x\df t,
\end{split}
\end{equation*}
and
\begin{eqnarray*}
I_3(Q_\e)
&&=2(2-\al)^4s^3\la^4\iint_{Q_\e}\xi^3v^2\psi_\e^{4-2\al}|\nabla\psi_\e|^4\df x\df t+(2-\al)^4s^3\la^3\iint_{Q_\e}\xi^3v^2\psi_\e^{2-\al}|\nabla\psi_\e|^4\df x\df t\\
&&\hspace{4.5mm}+2(2-\al)^3s^3\la^3\iint_{Q_\e}\xi^3v^2\psi_\e^{3-\al}(D^2\psi_\e\nabla\psi_\e)\cdot\nabla\psi_\e\df x\df t\\
&& =2(2-\al)^4s^3\la^4\iint_{Q_\e}\xi^3v^2\psi_\e^{4-2\al}|\nabla\psi_\e|^4\df x\df t+(2-\al)^4s^3\la^3\iint_{Q_\e}\xi^3v^2\psi_\e^{2-\al}|\nabla\psi_\e|^4\df x\df t\\
&&\hspace{4.5mm}+2(2-\al)^3s^3\la^3\iint_{Q_\e}\xi^3v^2\psi_\e^{2-\al}|\nabla\psi_\e|^2 \left[\f{3}{16\e^{6}}(\e^2-|x|^2)(3\e^4+6\e^2|x|^2-|x|^4)\right]\df x\df t\\
&&\geq 2(2-\al)^4s^3\la^4\iint_{Q_\e}\xi^3v^2\psi_\e^{4-2\al}|\nabla\psi_\e|^4\df x\df t+(2-\al)^4s^3\la^3\iint_{Q_\e}\xi^3v^2\psi_\e^{2-\al}|\nabla\psi_\e|^4\df x\df t.
\end{eqnarray*}
From the 8th term in $I_4(Q_\e)$,
\begin{equation*}
\begin{split}
&(2-\al)^3s\la^3\iint_{Q_\e}\xi \psi_\e^{3-\al}v(\nabla v\cdot \nabla\psi_\e)|\nabla \psi_\e|^2\df x\df t\\
&\geq -(2-\al)^2s\la^2\iint_{Q_\e}\xi\psi_\e^2(\nabla v\cdot \nabla\psi_\e)^2\df x\df t-Cs\la^4\iint_{Q_\e}\xi^3 v^2 \psi_\e^{4-2\al}|\nabla\psi_\e|^4\df x\df t,
\end{split}
\end{equation*}
and from the 9th-10th terms in $I_4(Q_\e)$,
\begin{equation*}
\begin{split}
&(2-\al)^2(3-\al)s\la^2\iint_{Q_\e}\xi \psi_\e v(\nabla v\cdot \nabla\psi_\e)|\nabla \psi_\e|^2\df x\df t\\
&+(N+2)(2-\al)^2s\la^2\iint_{Q_\e}\xi\psi_\e^2 v\left(\f{3}{2\e}-\f{1}{2\e^3}|x|^2\right)(\nabla v\cdot \nabla \psi_\e)\df x\df t\\
&-3(2-\al)^2s\la^2\iint_{Q_\e}\xi \psi_\e^2 v\f{1}{\e^3}|x|^2(\nabla v\cdot \nabla \psi_\e)\df x\df t\\
&\geq -Cs\la\iint_{Q_\e}\xi \psi_\e^2 (\nabla v\cdot \nabla \psi_\e)^2\df x\df t+Cs\la^3\iint_{Q_\e} \xi v^2\df x\df t
\end{split}
\end{equation*}
by \eqref{12.16.1}  and \eqref{12.18.1},
and from the 11th term in $I_4(Q_\e)$,
\begin{eqnarray*}
&&(2-\al)s\la\iint_{Q_\e}\xi \psi_\e^\al v (\nabla v\cdot x)\f{3}{8\e^6}\left[5(2+N)\e^4-6(4+N)\e^2|x|^2+(6+N)|x|^4\right]\df x\df t\\
&&\geq -Cs\la\iint_{Q_\e}\xi \psi_\e^\al |v| |\nabla v||\nabla\psi_\e|\f{1}{\f{3}{2\e}-\f{1}{2\e^3}|x|^2}\e^{-2}\df x\df t\\
&&\geq -C\la\iint_{Q_\e} \xi \psi_\e^\al |\nabla v|^2|\nabla\psi_\e|^2\df x\df t-Cs^2\la\e^{\al-2}\iint_{Q_\e} \xi  v^2\df x\df t
\end{eqnarray*}
by \eqref{01.17.3}.
Thus, there exist $\la_0\geq 1$ such that for all $\la\geq \la_0$ and for all $s\geq 1$, we have
\begin{eqnarray*}
&&I_4(Q_\e)\\
&&=2(2-\al)^2s\la^2\iint_{Q_\e}\xi\psi_\e^2(\nabla v\cdot \nabla\psi_\e)^2\df x\df t\\
&&\hspace{4.5mm}+2(2-\al)s\la\iint_{Q_\e} \xi \psi_\e^\al (\nabla v\cdot \nabla \psi_\e)^2\df x\df t-2(2-\al)s\la\iint_{Q_\e}\xi \psi_\e^{\al+1}\f{1}{\e^3}(\nabla v\cdot x)^2\df x\df t\\
&&\hspace{4.5mm}-\al(2-\al)s\la\iint_{Q_\e}\xi\psi_\e^\al |\nabla v|^2|\nabla\psi_\e|^2\df x\df t+2(2-\al)s\la\iint_{Q_\e}\xi \psi_\e^{\al+1}\left(\f{3}{2\e}-\f{1}{2\e^3}|x|^2\right)|\nabla v|^2\df x\df t\\
&&\hspace{4.5mm}+(2-\al)^3s\la^3\iint_{Q_\e}\xi \psi_\e^{3-\al}v(\nabla v\cdot \nabla \psi_\e)|\nabla\psi_\e|^2\df x\df t\\
&&\hspace{4.5mm}+(2-\al)^2(3-\al)s\la^2\iint_{Q_\e}\xi \psi_\e v(\nabla v\cdot \nabla\psi_\e)|\nabla\psi_\e|^2\df x\df t\\
&&\hspace{4.5mm}+(N+2)(2-\al)^2s\la^2\iint_{Q_\e}\xi\psi_\e^2 v\left(\f{3}{2\e}-\f{1}{2\e^3}|x|^2\right)(\nabla v\cdot \nabla \psi_\e)\df x\df t\\
&&\hspace{4.5mm}-3(2-\al)^2s\la^2\iint_{Q_\e}\xi \psi_\e^2 v\f{1}{\e^3}|x|^2(\nabla v\cdot \nabla \psi_\e)\df x\df t\\
&&\hspace{4.5mm}+(2-\al)s\la\iint_{Q_\e}\xi \psi_\e^\al v (\nabla v\cdot x)\f{3}{8\e^6}\left[5(2+N)\e^4-6(4+N)\e^2|x|^2+(6+N)|x|^4\right]\df x\df t\\
&&\geq +(2-\al)^2s\la^2\iint_{Q_\e}\xi\psi_\e^2(\nabla v\cdot \nabla\psi_\e)^2\df x\df t+(2-\al)^2s\la\iint_{Q_\e}\xi\psi_\e^\al |\nabla v|^2|\nabla\psi_\e|^2\df x\df t\\
&&\hspace{4.5mm}-Cs\la^4\iint_{Q_\e}\xi^3v^2\psi_\e^{4-2\al}|\nabla\psi_\e|^4\df x\df t-Cs\la^3\iint_{Q_\e}\xi v^2\df x\df t -Cs\la\e^{\al-2}\iint_{Q_\e}\xi v^2\df x\df t,
\end{eqnarray*}
where the 2nd-3rd terms in the equality use \eqref{01.17.2}, and the 4th-5th terms used \eqref{01.17.1}.
Overall, we have
\begin{eqnarray*}
&&(P_1v,P_2v)_{L^2(Q_\e)}\\
&&\geq 2(2-\al)^4s^3\la^4\iint_{Q_\e}\xi^3v^2\psi_\e^{4-2\al}|\nabla\psi_\e|^4\df x\df t+(2-\al)^4s^3\la^3\iint_{Q_\e}\xi^3v^2\psi_\e^{2-\al}|\nabla\psi_\e|^4\df x\df t\\
&&\hspace{4.5mm} +(2-\al)^2s\la^2\iint_{Q_\e}\xi\psi_\e^2(\nabla v\cdot \nabla\psi_\e)^2\df x\df t+(2-\al)^2s\la\iint_{Q_\e}\xi\psi_\e^\al |\nabla v|^2|\nabla\psi_\e|^2\df x\df t\\
&&\hspace{4.5mm}-Cs\la^4\iint_{Q_\e}\xi^3v^2\psi_\e^{4-2\al}|\nabla\psi_\e|^4\df x\df t-Cs^2\la^3\e^{\al-2}\iint_{Q_\e}\xi v^2\df x\df t
\end{eqnarray*}
where the constants depend only on $\al, R, T, N$, and $\Om$.

\textbf{ 4)~Overall Estimation.}  From estimations  {\bf { 1)-3)} } in this section, we have
\begin{eqnarray*}
&&\|P_1v\|_{L^2(Q)}^2+\|P_2v\|_{L^2(Q)}^2+s\la\iint_{\pt Q}\xi \psi_\e^{2\al}\left|\f{\pt v}{\pt \nu}\right|^2\left|\f{\pt \eta}{\pt \nu}\right|\df S\df t\\
&&+s\la^2\iint_{Q} \xi (\psi_\e^\al\nabla v\cdot \nabla\eta)^2\df x\df t+s^3\la^4\iint_Q \xi^3v^2(\psi_\e^\al |\nabla\eta|^2)^2\df x\df t\\
&&+ s\la\iint_{Q_{2R}}\xi \psi_\e^\al |\nabla v|^2|\nabla\psi_\e|^2\df x\df t+s^3\la^3\iint_{Q_{2R}}\xi^3v^2\psi_\e^{2-\al}|\nabla \psi_\e|^4\df x\df t\\
&&\leq C\|P_1v\|_{L^2(Q)}^2+C\|P_2v\|_{L^2(Q)}^2+Cs\la\iint_{\pt Q}\xi \psi_\e^{2\al}\left|\f{\pt v}{\pt \nu}\right|^2\left|\f{\pt \eta}{\pt \nu}\right|\df S\df t\\
&&\hspace{4.5mm}Cs\la^2\iint_{Q_{2R}^c}\xi (\psi_\e^\al \nabla v\cdot \nabla\eta)^2\df x\df t+Cs^3\la^4\iint_{Q_{2R}^c}\xi ^3v^2|\nabla\eta|^4\df x\df t\\
&&\hspace{4.5mm}+Cs\la^2\iint_{Q_{\e,2R}}\xi (\nabla v\cdot x)^2\df x\df t+Cs\la\iint_{Q_{\e,2R}}\xi |x|^\al |\nabla v|^2\df x\df t\\
&&\hspace{4.5mm}+Cs^3\la^4\iint_{Q_{\e,2R}}\xi^3v^2|x|^{4-2\al}\df x\df t+Cs^3\la^3\iint_{Q_{\e,2R}}\xi^3v^2|x|^{2-\al}\df x\df t\\
&&\hspace{4.5mm}+Cs\la^2\iint_{Q_\e}\xi \psi_\e^2(\nabla v\cdot \nabla\psi_\e)^2\df x\df t+Cs\la\iint_{Q_\e}\xi \psi_\e^\al |\nabla v|^2|\nabla \psi_\e|^2\df x\df t\\
&&\hspace{4.5mm}+Cs^3\la^4\iint_{Q_\e}\xi^3v^2\psi_\e^{4-2\al}|\nabla\psi_\e|^4\df x\df t+Cs^3\la^3\iint_{Q_\e}\xi^3v^2\psi_\e^{2-\al}|\nabla\psi_\e|^4\df x\df t\\
&&\leq C\|e^{-s\si}g\|{L^2(Q)}^2+Cs^3\la^4\iint_{\wh\om\ts (0,T)}\xi^3v^2\df x\df t+Cs\la\iint_{Q_{2R}^c}\xi |\nabla v|^2\df x\df t\\
&&\hspace{4.5mm}+Cs^2\la^3\e^{2-\al}\iint_{Q_\e}\xi^3 v^2\df x\df t.
\end{eqnarray*}

\textbf{5)~ Estimation of $v_t$ and $\Div(\psi_\e^\al \nabla v)$.}
By the definition of $P_1v$ and $P_2v$ of \dref{gbz3}, we have
\begin{equation*}
\begin{split}
&s^{-1}\iint_{Q}\xi^{-1}|v_t|^2\df x\df t\\
&\leq Cs^{-1}\|P_1v\|_{L^2(Q)}^2+Cs\iint_{Q}\xi^{-1}\psi_\e^{2\al}(\nabla v\cdot \nabla \si)^2\df x\df t+Cs\iint_Q \xi^{-1}\big|\Div(v\psi_\e^\al\nabla \si)\big|^2\df x\df t\\
&\leq Cs^{-1}\|P_1v\|_{L^2(Q)}^2+Cs\la^2\iint_Q \xi (\psi_\e^\al \nabla v\cdot \nabla\eta)^2\df x\df t+Cs\la^4\iint_Q \xi^3 v^2(\psi_\e^\al |\nabla\eta|^2)^2\df x\df t\\
&\hspace{4.5mm}+Cs\la^2\iint_{Q_{2R}^c}\xi v^2\df x\df t+Cs\la^2\iint_{Q_{\e,2R}}\xi v^2\df x\df t+Cs\la^2\iint_{Q_\e}\xi v^2\df x\df t\\
&\leq Cs^{-1}\|P_1v\|_{L^2(Q)}^2+Cs\la^2\iint_Q \xi (\psi_\e^\al \nabla v\cdot \nabla\eta)^2\df x\df t+Cs\la^4\iint_Q \xi^3 v^2(\psi_\e^\al |\nabla\eta|^2)^2\df x\df t\\
&\hspace{4.5mm}+Cs\la^2\iint_{Q_{2R}^c} \xi v^2\df x\df t+Cs\la^2\iint_{Q_\e}\xi v^2\df x\df t\\
&\hspace{4.5mm}+C\la \iint_{Q_{\e,2R}} \xi\psi_\e^\al |\nabla v|^2\df x\df t+Cs^2\la^3\iint_{Q_{\e,2R}} \xi^3 v^2|x|^{2-\al}\df x\df t\\
&\hspace{4.5mm}+C\la \iint_{Q_{2R}^c\cup Q_\e} \xi\psi_\e^\al |\nabla v|^2|\nabla \psi_\e|^2\df x\df t+C\la^3\iint_{Q_{2R}^c\cup Q_\e}\xi^3 v^2\psi_\e^{2-\al}|\nabla\psi_\e|^4\df x\df t\\
&\hspace{4.5mm}+C\la \iint_{Q_{2R}^c}\xi  v^2\df x\df t+C\la \e^{\al-2}\iint_{Q_\e}\xi v^2\df x\df t
\end{split}
\end{equation*}
by the estimate $I_4(Q_{\e,2R})$.
Similarly,
\begin{eqnarray*}
&&s^{-1}\iint_Q \xi^{-1}\big|\Div(\psi_\e^\al \nabla v)\big|^2\df x\df t\\
&&\leq Cs^{-1}\|P_2v\|_{L^2(Q)}^2+Cs^3\iint_Q \xi^{-1} v^2\psi_\e^{2\al} |\nabla\si|^4\df x\df t+Cs\iint_{Q} \xi^{-1}|\si_t|^2v^2\df x\df t\\
&&\leq Cs^{-1}\|P_2v\|_{L^2(Q)}^2+Cs^3\la^4\iint_Q \xi^3v^2(\psi_\e^\al|\nabla\eta|^2)^2\df x\df t+Cs\iint_Q \xi^\f{3}{2} v^2\df x\df t\\
&&\leq +Cs^{-1}\|P_2v\|_{L^2(Q)}^2+Cs^3\la^4\iint_Q \xi^3v^2(\psi_\e^\al|\nabla\eta|^2)^2\df x\df t \\
&&\hspace{4.5mm}+Cs\iint_{Q_{2R}^c} \xi^3v^2\df x\df t+Cs\iint_{Q_\e}\xi^\f{3}{2}v^2\df x\df t\\
&&\hspace{4.5mm}+C\iint_{Q_{\e,2R}} \xi\psi_\e^\al |\nabla v|^2\df x\df t+C(\la^2+s^2)\iint_{Q_{\e,2R}} \xi^3 v^2|x|^{2-\al}\df x\df t\\
&&\hspace{4.5mm}+C\iint_{Q_{2R}^c\cup Q_\e} \xi\psi_\e^\al |\nabla v|^2|\nabla \psi_\e|^2\df x\df t+C\la^2\iint_{Q_{2R}^c\cup Q_\e}\xi^3 v^2\psi_\e^{2-\al}|\nabla\psi_\e|^4\df x\df t\\
&&\hspace{4.5mm}+C\iint_{Q_{2R}^c}\xi  v^2\df x\df t+C\e^{\al-2}\iint_{Q_\e}\xi v^2\df x\df t
\end{eqnarray*}
by the estimate $I_1(Q_{\e,2R})$.

\textbf{6)~ From (4) and (5), we obtain}
\begin{eqnarray*}
		&&s^{-1}\iint_Q \xi^{-1}\left(|v_t|^2+|\Div(\psi_\e^\al \nabla v)|^2\right)\df x\df t +s\la\iint_{\pt Q}\xi \psi_\e^{2\al}\left|\f{\pt v}{\pt \nu}\right|^2\left|\f{\pt \eta}{\pt \nu}\right|\df S\df t\\
		&&+s\la^2\iint_{Q} \xi (\psi_\e^\al\nabla v\cdot \nabla\eta)^2\df x\df t+s^3\la^4\iint_Q \xi^3v^2(\psi_\e^\al |\nabla\eta|^2)^2\df x\df t\\
		&&+ s\la\iint_{Q_{2R}}\xi \psi_\e^\al |\nabla v|^2|\nabla\psi_\e|^2\df x\df t+s^3\la^3\iint_{Q_{2R}}\xi^3v^2\psi_\e^{2-\al}|\nabla \psi_\e|^4\df x\df t\\
		&&\leq C\|e^{-s\si}g\|_{L^2(Q)}^2+Cs^3\la^4\iint_{\wh\om\ts (0,T)}\xi^3v^2\df x\df t+Cs\la\iint_{Q_{2R}^c}\xi |\nabla v|^2\df x\df t\\
		&&\hspace{4.5mm}+Cs^2\la^3\e^{\al-2}\iint_{Q_\e}\xi^3 v^2\df x\df t.
\end{eqnarray*}

\textbf{ 7)~ We choose $\zeta\in C^\iy(\R^N), 0\leq \zeta\leq 1$ satisfying}
\begin{equation*}
	\zeta=1 \mbox{ on } \R^N\se B_{2R}, \quad \zeta=0 \mbox{ on } B_R, \quad |\nabla\zeta|\leq \f{C}{R} \mbox{ on } \R^N.
\end{equation*}
Then,  we have
\begin{eqnarray*}
		&&s\la\iint_{Q_{2R}^c} \xi \psi_\e^\al |\nabla v|^2\df x\df t\\
		&&\leq +s\la\iint_Q \nabla v\cdot (\xi \zeta \psi_\e^\al\nabla  v)\df x\df t=-s\la\iint_Q v\Div(\xi \zeta \psi_\e^\al \nabla v)\df x\df t\\
		&&\leq -s\la \iint_Q \xi v\psi_\e^\al \nabla \zeta\cdot \nabla v\df x\df t-s\la^2\iint_Q \zeta \xi v\psi_\e^\al \nabla v\cdot \nabla\eta\df x\df t-s\la\iint_Q \zeta \xi v\Div(\psi_\e^\al \nabla v)\df x\df t\\
		&&\leq +Cs^\f{1}{2}\la \iint_{Q_{\e,2R}}\xi \psi_\e^\al |\nabla v|^2\df x\df t+Cs^\f{3}{2}\la\iint_Q \xi v^2\psi_\e^\al |\nabla\zeta|^2\df x\df t\\
		&&\hspace{4.5mm}+Cs^2\la^2\iint_Q \zeta\xi v^2 \df x\df t+C\la^2\iint_Q\zeta \xi (\psi_\e^\al \nabla v\cdot \nabla\eta)^2\df x\df t\\
		&&\hspace{4.5mm}+Cs^{-1}\la^{-\f{1}{2}}\iint_Q \zeta \xi^{-1}\big|\Div(\psi_\e^\al \nabla v)\big|^2\df x\df t+Cs^3\la^\f{5}{2}\iint_Q \zeta \xi^3v^2\df x\df t\\
		&&\leq Cs^\f{1}{2}\la \iint_{Q_{\e, 2R}}\xi \psi_\e^\al |\nabla v|^2\df x\df t+C\la^2\iint_Q\zeta \xi (\psi_\e^\al \nabla v\cdot \nabla\eta)^2\df x\df t\\
		&&\hspace{4.5mm}+ Cs^{-1}\la^{-\f{1}{2}}\iint_Q \zeta \xi^{-1}\big|\Div(\psi_\e^\al \nabla v)\big|^2\df x\df t+C\f{1}{R^{4-2\al}}s^\f{3}{2}\la\iint_{Q_{\e,2R}} \xi^3 v^2\psi_\e^{2-\al}|\nabla\psi_\e|^4\df x\df t\\
		&&\hspace{4.5mm}+C\f{1}{R^{2-\al}}s^3\la^\f{5}{2}\iint_{Q_{\e,2R}} \xi^3 v^2\psi_\e^{2-\al}|\nabla\psi_\e|^4\df x\df t+Cs^3\la^\f{5}{2}\int_0^T\int_{\Om\se (B_{2R}\cup \wh\om)} \xi^3v^2(\psi_\e^\al |\nabla\eta|^2)^2\df x\df t\\
		&&\hspace{4.5mm}+Cs^3\la^\f{5}{2}\iint_{\wh\om\ts (0,T)}\xi^3v^2\df x\df t.
\end{eqnarray*}
Hence,
\begin{equation}\label{01.18.1}
	\begin{split}
		&s^{-1}\iint_Q \xi^{-1}\left(|v_t|^2+|\Div(\psi_\e^\al \nabla v)|^2\right)\df x\df t +s\la\iint_{\pt Q}\xi \psi_\e^{2\al}\left|\f{\pt v}{\pt \nu}\right|^2\left|\f{\pt \eta}{\pt \nu}\right|\df S\df t\\
		&+s\la^2\iint_{Q} \xi (\psi_\e^\al\nabla v\cdot \nabla\eta)^2\df x\df t+s^3\la^4\iint_Q \xi^3v^2(\psi_\e^\al |\nabla\eta|^2)^2\df x\df t\\
		&+ s\la\iint_{Q}\xi \psi_\e^\al |\nabla v|^2|\nabla\psi_\e|^2\df x\df t+s^3\la^3\iint_{Q_{2R}}\xi^3v^2\psi_\e^{2-\al}|\nabla \psi_\e|^4\df x\df t\\
		&\leq C\|e^{-s\si}g\|_{L^2(Q)}^2+Cs^3\la^4\iint_{\wh\om\ts (0,T)}\xi^3v^2\df x\df t+Cs^2\la^3\e^{\al-2}\iint_{Q_\e}\xi^3 v^2\df x\df t.
	\end{split}
\end{equation}

\textbf{ 8)~ By \eqref{01.18.1}, we obtain}
\begin{eqnarray}\label{06.10.1}
		&&s\la^2\iint_{Q} \xi (\psi_\e^\al\nabla v\cdot \nabla\eta)^2\df x\df t+s^3\la^4\iint_Q \xi^3v^2(\psi_\e^\al |\nabla\eta|^2)^2\df x\df t\crr
		&&+ s\la\iint_{Q}\xi \psi_\e^\al |\nabla v|^2|\nabla\psi_\e|^2\df x\df t+s^3\la^3\iint_{Q_{2R}}\xi^3v^2\psi_\e^{2-\al}|\nabla \psi_\e|^4\df x\df t\crr
		&&\leq C\|e^{-s\si}g\|_{L^2(Q)}^2+Cs^3\la^4\iint_{\wh\om\ts (0,T)}\xi^3v^2\df x\df t+Cs^2\la^3\e^{\al-2}\iint_{Q_\e}\xi^3 v^2\df x\df t.
\end{eqnarray}
 Returning to the solutions of \eqref{06.08.1}, we note that
\begin{equation*}
\nabla u = e^{s\si} \left[\nabla v + sv\nabla\si\right] \quad \Longleftrightarrow \quad e^{-s\si}\nabla u = \nabla v - s\la \xi v\nabla\eta.
\end{equation*}
Consequently, from \eqref{06.10.1}, we derive the following inequality:
\begin{eqnarray*}
&& s\la \iint_{Q} \xi \psi_\e^\al |\nabla u|^2 |\nabla\psi_\e|^2 e^{-2s\si} \, \mathrm{d}x \, \mathrm{d}t + s^3\la^4 \iint_Q \xi^3 u^2 (\psi_\e^\al |\nabla\eta|^2)^2 e^{-2s\si} \, \mathrm{d}x \, \mathrm{d}t \\
&&\leq C \|e^{-s\si} g\|_{L^2(Q)}^2 + Cs^3\la^4 \iint_{\wh\om \times (0,T)} \xi^3 u^2 e^{-2s\si} \, \mathrm{d}x \, \mathrm{d}t + Cs^2\la^3 \e^{\al-2} \iint_{Q_\e} \xi^3 u^2 e^{-2s\si} \, \mathrm{d}x \, \mathrm{d}t,
\end{eqnarray*}
where the positive constant $C$ depends solely on $\al$, $R$, $T$, $N$, and $\Om$.
This leads us to the following Theorem \ref{06.10.T1}.

\begin{theorem}\label{06.10.T1}
Let $T > 0$, and let $\Om \subset \R^N$ be a $C^3$ bounded domain with $0 \in \Om$. Let $R > 0$ be such that $B_{4R} \subset \Om$, and let $\wh\om = B_{2R}(x_0) \subset B_{3R}(x_0) \subset \Om$ be an open subset satisfying $B_{3R} \cap B_{3R}(x_0) = \varnothing$. Let $\eta$ be defined as in Lemma \ref{12.26.L1}. Then, there exist positive constants $C$, $\la_0 \geq 1$, and $s_0 \geq 1$, depending only on $\al$, $R$, $T$, $N$, and $\Om$, such that for every solution $u(\cdot)$ of \eqref{06.08.1}, and for all $s \geq s_0$ and $\la \geq \la_0$, the following inequality holds:
\begin{equation} \label{06.10.2}
\begin{split}
& s\la \iint_{Q} \xi \psi_\e^\al |\nabla u|^2 |\nabla\psi_\e|^2 e^{-2s\si} \, \mathrm{d}x \, \mathrm{d}t + s^3\la^4 \iint_Q \xi^3 u^2 (\psi_\e^\al |\nabla\eta|^2)^2 e^{-2s\si} \, \mathrm{d}x \, \mathrm{d}t \\
&\leq C \|e^{-s\si} g\|_{L^2(Q)}^2 + Cs^3\la^4 \iint_{\wh\om \times (0,T)} \xi^3 u^2 e^{-2s\si} \, \mathrm{d}x \, \mathrm{d}t + Cs^2\la^3 \e^{\al-2} \iint_{Q_\e} \xi^3 u^2 e^{-2s\si} \, \mathrm{d}x \, \mathrm{d}t.
\end{split}
\end{equation}
\end{theorem}
\begin{remark}\label{06.10.R1}
In Theorem \ref{06.10.T1}, the functions $\eta$, $u$, $\xi$, and $\si$ are dependent on the parameter $\e > 0$.
\end{remark}

\subsection{Approximation}

 Let us denote $u_k = u_\varepsilon$ and $Q_k = Q_\varepsilon$ for $\varepsilon = \frac{1}{k}$.
Assume that $u_k$ ($k \in \mathbb{N}$) is the solution of \eqref{06.08.1}, and $u(\cdot)$ is the solution of \eqref{06.08.2} on $Q = (-\beta, T + \beta)$ with $u_k(T) = u(T) = u_T \in H^1(\Omega)$. Since $B_R \times (0, T)$ is a compact subset of $\Omega \times (-\beta, T + \beta)$, by Theorem 3.11 or  \cite[Theorem 3.14]{FF} and Lemma \ref{01.18.L1}, there exists a subsequence of $\{u_k\}_{k \in \mathbb{N}}$, which we still denote by the same notation, such that
\begin{equation*}
u_k \to u \quad \text{uniformly on } B_R \times (0, T).
\end{equation*}
Consequently, for any $h > 0$, there exists $k_0 > 0$ such that for all $k \geq k_0$,
\begin{equation*}
|u_k(x, t) - u(x, t)| < h \quad \text{on } B_R \times (0, T).
\end{equation*}
As $k \to +\infty$, we have
\begin{eqnarray}\label{06.10.3}
&&k^{2 - \alpha}\iint_{Q_k}\xi^3 u_k^2e^{-2s\sigma} \, dx \, dt\crr
&&\leq 2k^{2 - \alpha}\iint_{Q_k}\xi^3 (u_k - u)^2e^{-2s\sigma} \, dx \, dt + 2k^{2 - \alpha}\iint_{Q_k}\xi u^2e^{-2s\sigma} \, dx \, dt\crr
&&\leq Ck^{2 - \alpha}|B_k|h^2\int_0^T\Theta^3 e^{9\lambda|\eta|_\infty} e^{-2s\Theta e^{\lambda|\eta|_\infty}} \, dt \crr
&&\quad + Ck^{2 - \alpha}|B_k|\int_0^T\Theta^3 e^{9\lambda|\eta|_\infty} e^{-2s\Theta e^{\lambda|\eta|_\infty}}\left(\frac{1}{|B_k|}\int_{B_k} u^2 \, dx\right) \, dt \to 0,
\end{eqnarray}
where the convergence follows from the facts that $N \geq 2$, $\Theta^3 e^{9\lambda|\eta|_\infty} e^{-2s\Theta e^{\lambda|\eta|_\infty}}$ is bounded on $t \in (0, +\infty)$, and by the Lebesgue differentiation theorem,
\begin{equation*}
\lim_{k \to \infty}\frac{1}{|B_k|}\int_{B_k}u^2 \, dx = u(0, t) \quad \text{for all } t \in (0, T).
\end{equation*}
Let $\gamma \in (0, R)$ be given. By the definitions of $\eta$, $\xi$, and $\sigma$, for $\varepsilon \in (0, \frac{1}{2}\gamma)$, we know that
\begin{equation*}
\eta = |x|^\alpha \quad \text{for } x \in B_{2R} \setminus B_\gamma.
\end{equation*}
Hence, $\eta$, $\xi$, and $\sigma$ are independent of $\varepsilon \in (0, \frac{1}{2}\gamma)$ on $\Omega \setminus B_\gamma$. Note that $\eta$, $\psi_\varepsilon^\alpha$, $\nabla\psi_\varepsilon$, $\xi e^{-2s\sigma}$, and $\xi^3e^{-2s\sigma}$ are bounded continuous functions on $Q$, and $\eta(x) \to |x|^{2 - \alpha}$ everywhere on $B_{2R}$ as $\varepsilon \to 0$. By Theorem \ref{08.16.T1}, Corollary \ref{08.23.C1}, and the fact that the constants $C > 0$ in Theorem \ref{06.10.T1} are independent of $\varepsilon \in (0, \frac{1}{2}\gamma)$, we obtain
\begin{equation*}
\begin{split}
& s\lambda\iint_{Q_\gamma^c}\widehat{\xi} \psi^\alpha |\nabla u|^2e^{-2s\widehat{\sigma}}\, dx\, dt + s^3\lambda^4\iint_{Q_\gamma^c} \widehat{\xi}^3u^2\left(\psi^\alpha \left|\nabla \widehat{\eta}\right|^2\right)^2e^{-2s\widehat{\sigma}}\, dx\, dt\\
&\leq C\|e^{-s\widehat{\sigma}}g\|{L^2(Q)}^2 + Cs^3\lambda^4\iint_{\omega \times (0, T)}\widehat{\xi}^3u^2e^{-2s\widehat{\sigma}}\, dx\, dt,
\end{split}
\end{equation*}
where
\begin{equation}\label{06.10.4}
\widehat{\eta}(x)=
\begin{cases}
\eta(x), &x \in \Omega \setminus B_{2R},\\
|x|^{2 - \alpha}, &x \in B_{2R},
\end{cases}\quad \widehat{\xi}=\Theta e^{\lambda(8|\widehat{\eta}|_\infty+\widehat{\eta}(x))},\quad \widehat{\sigma}=\Theta(t)e^{10\lambda|\widehat{\eta}|_\infty}-\widehat{\xi}(x, t).
\end{equation}
Letting $\gamma \to 0$, we get
\begin{equation}\label{06.10.5}
\begin{split}
& s\lambda\iint_{Q}\widehat{\xi} \psi^\alpha |\nabla u|^2e^{-2s\widehat{\sigma}}\, dx\, dt + s^3\lambda^4\iint_{Q} \widehat{\xi}^3u^2\left(\psi^\alpha \left|\nabla \widehat{\eta}\right|^2\right)^2e^{-2s\widehat{\sigma}}\, dx\, dt\\
&\leq C\|e^{-s\widehat{\sigma}}g\|{L^2(Q)}^2 + Cs^3\lambda^4\iint_{\omega \times (0, T)}\widehat{\xi}^3u^2e^{-2s\widehat{\sigma}}\, dx\, dt,
\end{split}
\end{equation}
where the constants $C > 0$ depend only on $\alpha$, $R$, $T$, $N$, and $\Omega$.
Finally, for any $u_T \in L^2(\Omega)$, there exists a sequence $\{u_T^n\}_{n \in \mathbb{N}} \subset H^1(\Omega)$ such that $u_T^n \to u_T$ strongly in $L^2(\Omega)$. Let $u_n$ be the solution of \eqref{06.08.2} with $g_n = g$ and $u_n(T) = u_T^n$. Denote $U_n = u_n - u$. Then $U_n$ is the solution of the following equation
\begin{equation*}
\begin{cases}
\partial_t U_n + \Div(wU_n) = 0, &\text{in } Q,\\
U_n = 0, &\text{on } \partial Q, \\
U_n(T) = u_T^n - u_T, &\text{in } \Omega,
\end{cases}
\end{equation*}
and
\begin{equation*}
\iint_Q |\nabla U_n|^2w\, dx\, dt + \frac{1}{2}\int_\Omega U_n^2(0)\, dx\, dt = \frac{1}{2}\int_\Omega U_n^2(T)\, dx\, dt \to 0 \quad \text{as } n \to \infty.
\end{equation*}
This implies that (by Lemma \ref{08.15.L1})
\begin{equation*}
u_n \to u \quad \text{strongly in } L^2(0, T; H_0^1(\Omega; w)),
\end{equation*}
and then \eqref{06.10.5} holds for $u(T) = u_T \in L^2(\Omega)$.

We have thus obtained the following Theorem \ref{06.10.T2}.

\begin{theorem}\label{06.10.T2}
Let $T > 0$, and $\Omega \subset \mathbb{R}^N$ be a $C^3$ bounded domain with $0 \in \Omega$. Let $R > 0$ be such that $B_{4R} \subset \Omega$, $B_{3R}(x_0) \subset \omega$, and $B_{3R} \cap B_{3R}(x_0) = \varnothing$. Let $\widehat{\eta}$, $\widehat{\xi}$, $\widehat{\sigma}$ be defined in \eqref{06.10.4}. Then there exist constants $C > 0$ and $s_0 \geq 1$, $\lambda_0 \geq 1$, depending only on $\alpha$, $R$, $T$, $N$, and $\Omega$, such that for every solution $u(\cdot)$ of \eqref{06.08.2}, we have
\begin{equation*}
\begin{split}
& s\lambda\iint_{Q}\widehat{\xi} \psi^\alpha |\nabla u|^2e^{-2s\widehat{\sigma}}\, dx\, dt + s^3\lambda^4\iint_{Q} \widehat{\xi}^3u^2\left(\psi^\alpha \left|\nabla \widehat{\eta}\right|^2\right)^2e^{-2s\widehat{\sigma}}\, dx\, dt\\
&\leq C\|e^{-s\widehat{\sigma}}g\|{L^2(Q)}^2 + Cs^3\lambda^4\iint_{\omega \times (0, T)}\widehat{\xi}^3u^2e^{-2s\widehat{\sigma}}\, dx\, dt.
\end{split}
\end{equation*}
\end{theorem}

\section{Observability inequality}\label{S5}

 Now, let us set $g = 0$ in equation \eqref{06.08.2}. Upon multiplying both sides of \eqref{06.08.2} by $u(\cdot)$ and integrating over the domain $\Omega \times (0, t)$, we arrive at the following identity:
\begin{equation*}
\begin{split}
\frac{1}{2}\int_{\Omega} u^2(t) \, \mathrm{d}x = \frac{1}{2}\int_{\Omega} u^2(0) \, \mathrm{d}x + \iint_{Q} |\nabla u|^2 w \, \mathrm{d}x \, \mathrm{d}t.
\end{split}
\end{equation*}
This indicates that $\int_{\Omega} u^2(t) \, \mathrm{d}x$ is a non-decreasing function on the interval $(0, T)$. Consequently, we can derive the following chain of inequalities:
\begin{eqnarray*}
\int_{\Omega} u^2(0) \, \mathrm{d}x
&\leq& \frac{2}{T}\int_{\frac{T}{4}}^{\frac{3T}{4}}\int_{\Omega} u^2(x, t) \, \mathrm{d}x \, \mathrm{d}t \leq \frac{C}{T}\int_{\frac{T}{4}}^{\frac{3T}{4}} \int_{\Omega} \psi^{\alpha} |\nabla u|^2 \, \mathrm{d}x \, \mathrm{d}t\\
&\leq& C\int_{\frac{T}{4}}^{\frac{3T}{4}}\int_{\Omega} \widehat{\xi} \psi^{\alpha} |\nabla u|^2 e^{-2s\widehat{\sigma}} \, \mathrm{d}x \, \mathrm{d}t \leq C\iint_{\omega \times (0, T)} \widehat{\xi}^3 u^2 e^{-2s\widehat{\sigma}} \, \mathrm{d}x \, \mathrm{d}t\\
&\leq& C\iint_{\omega \times (0, T)} u^2 \, \mathrm{d}x \, \mathrm{d}t.
\end{eqnarray*}
These inequalities are justified by the following facts:
\begin{equation*}
\begin{split}
\widehat{\xi} e^{-2s\widehat{\sigma}} \geq C   \text{ on } \Omega \times \left(\frac{T}{4}, \frac{3T}{4}\right), \quad \text{and} \quad \widehat{\xi}^3 e^{-2s\widehat{\sigma}} \leq C   \text{ on } Q,
\end{split}
\end{equation*}
where the second inequality employs Lemma \ref{08.15.L1}, and the fourth inequality utilizes Theorem \ref{06.10.T2}. Here, the constants $C > 0$ depend solely on $\alpha$, $R$, $T$, $N$, and $\Omega$. This leads us to the following Theorem \ref{06.10.T3}.

\begin{theorem}\label{06.10.T3}
Let $T > 0$, and let $\Omega \subset \mathbb{R}^N$ be a bounded domain of class $C^3$ with $0 \in \Omega$. Let $\omega \subset \Omega$ be a non-empty open subset such that $0 \notin \omega$. Then, there exists a constant $C > 0$, depending only on $\alpha$, $T$, $\omega$, and $\Omega$, such that for all solutions of \eqref{06.08.2} with $g = 0$, there is
\begin{equation*}
\|u(0)\|_{L^2(\Omega)}^2 \leq C\iint_{\omega \times (0, T)} u^2 \, \mathrm{d}x \, \mathrm{d}t.
\end{equation*}
\end{theorem}
We recall that the system \eqref{12.14.1} is said to be null controllable if, for any $\varphi_0 \in L^2(\Omega)$, there exists $f \in L^2(Q)$ such that
\begin{equation*}
\varphi(T) = 0 \quad \text{in } \Omega.
\end{equation*}
By applying the standard  Hilbert Uniqueness Method, we obtain the following Corollary \ref{06.10.C1}.
\begin{corollary}\label{06.10.C1}
The system \eqref{12.14.1} is null controllable.
\end{corollary}

\end{document}